\documentclass[11pt]{article}
\usepackage{graphicx}
\usepackage{amsmath,amsfonts,amsthm}
\usepackage{amssymb,latexsym}
\usepackage{fixmath}
\usepackage{mathrsfs,amsbsy}
\usepackage{dsfont}
\usepackage{enumerate}

\def\wtd{\widetilde}

\DeclareMathOperator{\diag}{diag}

\DeclareMathOperator{\F}{F}
\DeclareMathOperator{\HH}{H}
\DeclareMathOperator{\T}{T}

\def\ba{\pmb{a}}
\def\bb{\pmb{b}}
\def\bc{\pmb{c}}

\def\be{\pmb{e}}

\def\bm{\pmb{m}}

\def\bp{\pmb{p}}

\def\bt{\pmb{t}}

\def\bw{\pmb{w}}
\def\bx{\pmb{x}}
\def\by{\pmb{y}}

\newtheorem{theorem}{Theorem}[section]
\newtheorem{lemma}{Lemma}[section]
\newtheorem{corollary}{Corollary}[section]

\theoremstyle{definition}

\newtheorem{example}{Example}[section]

\numberwithin{equation}{section}
\numberwithin{figure}{section}
\numberwithin{table}{section}

\allowdisplaybreaks

\usepackage{kbordermatrix}

\usepackage{caption}
\usepackage{amsfonts,amsthm,yfonts}
\usepackage{amsmath,amssymb,enumerate,color}
\usepackage{algorithm,algorithmic}
\usepackage{epsfig,graphicx,psfrag,mathrsfs,subfigure}
\usepackage{eucal}
\usepackage{yhmath}
\usepackage{hyperref}
\usepackage{diagbox}
\RequirePackage{multirow}
 \usepackage{cleveref}
\usepackage{textcomp}
\usepackage[]{fontenc}
\usepackage{extarrows}
\def\wtd{\widetilde}

\usepackage{rotating}
\usepackage{lscape}
\usepackage{bm}
\usepackage{xurl}
\usepackage{xspace}
\usepackage{tikz}
\usetikzlibrary{petri} 
\usetikzlibrary{shadows,arrows,positioning,shapes.geometric} 

\usepackage{geometry}

\def\ba{\pmb{a}}
\def\bb{\pmb{b}}
\def\bc{\pmb{c}}

\def\be{\pmb{e}}

\def\bp{\pmb{p}}

\def\bt{\pmb{t}}

\def\bw{\pmb{w}}
\def\bx{\pmb{x}}
\def\by{\pmb{y}}

\def\bzs{\mathbf{0}}

\def\diag{{\rm diag}}

\def\scrR{\mathscr{R}}
\def\scrP{\mathscr{P}}

\def\BF{{\pmb{F}}}
\def\BW{{\pmb{W}_{\otimes}}}
\def\wtd{\widetilde}

\def\bbC{\mathbb{C}}

\def\bbP{\mathbb{P}}

\renewcommand{\algorithmicrequire}{\textbf{Input:}}
\renewcommand{\algorithmicensure}{\textbf{Output:}}

\numberwithin{equation}{section}
\numberwithin{figure}{section}
\numberwithin{table}{section}

\graphicspath{{figures/}{figure/}{pictures/}%
	{picture/}{pic/}{pics/}{image/}{images/}}

\title{Rational Minimax  Approximations for Matrix-Valued Functions: Existence, Optimality and Algorithms
}
 \author{Lei-Hong Zhang\thanks{Corresponding author. School of Mathematical Sciences, Soochow University, Suzhou 215006, Jiangsu, China. This work was
 supported in part by the National Natural Science Foundation of China (NSFC-12471356, NSFC-12371380).
        Email: {\tt longzlh@suda.edu.cn}.}\and Chenkun Zhang\thanks{School of Mathematical Sciences, Soochow University, Suzhou 215006, Jiangsu, China. Email: {\tt 19816896524@163.com}.}
        }
 
 \date{\today}

\begin{document}

\maketitle

\begin{abstract} 
In this paper, we study  rational minimax approximation for continuous complex matrix-valued functions in the Frobenius norm, where all approximant entries share a common denominator. This generalizes classical scalar rational approximation, with applications in system modeling, microwave design, and nonlinear eigenvalue problems. We first prove the existence of such matrix-valued approximants on point sets dense in themselves, extending Walsh's foundational scalar result. Next, we establish characterizations of the local and global minimax approximants by deriving primal/dual matrix-valued Kolmogorov criteria and a Ruttan-type sufficient condition for global optimality. For analytic functions on a continuum, we   link continuum minimax approximation to approximation on its boundary, and finite boundary samples via the maximum norm principle. We show that Ruttan's sufficient optimality condition provides a certificate under which a minimax approximant obtained from the boundary or from a discrete set of boundary nodes also solves the original continuum problem. Finally, for discrete approximation, we connect these conditions to a dual problem and the related dual-based numerical method {\sf m-d-Lawson}: when the original minimax problem admits a solution, strong duality is equivalent to Ruttan's sufficient optimality condition, and, the optimality equations underlying the {\sf m-d-Lawson} iteration coincide with Kolmogorov’s dual criteria. These results provide a theoretical basis for certifying and computing matrix-valued rational minimax approximants.
\end{abstract}

\medskip
{\small
\medskip

}

\tableofcontents 
 
\section{Introduction}\label{sec_intro}
Rational minimax approximation is one of central tools in approximation theory and numerical computation \cite{appr:web,drnt:2024,nast:2023,natr:2026,tref:2019a}, especially when the target function has poles, sharp transitions, oscillatory behavior, or spectral features that are poorly captured by polynomial approximants. In many modern applications, the data to be approximated are naturally matrix-valued rather than scalar-valued. Such matrix-valued rational approximations arise, for example, in linear time-invariant systems \cite{begu:2017,gogu:2021}, frequency-domain multiport rational modeling \cite{demd:2009,gogu:2021,gust:2006,gust:2009,guse:1999}, computer-aided design of microwave duplexers \cite{trai:2010,trms:2007,zhzy:2025}, and the numerical solution of nonlinear eigenvalue problems \cite{guti:2017,gupt:2022,limp:2022,saem:2020,zhgz:2026}. These applications motivate approximation methods that approximate all entries of a matrix-valued function simultaneously while enforcing a shared denominator, a structure that is essential in system and frequency-domain models.

Suppose $F:\mathscr{B}\subseteq \bbC\rightarrow \bbC^{s\times t}$  is a given continuous matrix-valued function with  
\begin{equation}\label{eq:Fx}
F(x)=\left[\begin{array}{ccc}f_{11}(x) &  \cdots & f_{1t}(x)   \\\ \vdots& \ddots & \vdots  \\ f_{s1}(x) &  \cdots & f_{st}(x)\end{array}\right]\in \bbC^{s\times t},~s\ge 1, t\ge 1,
\end{equation}
where each entry function $f_{ij}:\mathscr{B}\subseteq \bbC\rightarrow \bbC$ is continuous on a compact set $\mathscr{B}$ in the complex plane.   
Let $\|F(x)\|_{\F}$ be the Frobenius norm of $F(x)$ and  $\bbP_{n}$ be the set of complex polynomials $p$ with degree less than or equal to $n$, i.e., $\deg(p)\le n$. In \cite{zhzz:2025},  the following rational minimax approximation problem for $F$ is proposed:
\begin{equation}\label{eq:bestf0}
\eta_{\infty,\mathscr{B}}:=\inf_{R \in\scrR_{(s,t)}}\sup_{x\in \mathscr{B}}\|F(x)-R(x)\|_{\F}^2,
\end{equation}
where  
\begin{equation}\label{eq:rats}
\scrR_{(s,t)}:=\left\{R(x)=\left[\begin{array}{ccc}r_{11}(x) &  \cdots & r_{1t}(x)   \\\ \vdots& \ddots & \vdots  \\ r_{s1}(x) &  \cdots & r_{st}(x)\end{array}\right]  \Big| r_{ij}(x)=\frac{p_{ij}(x)}{q(x)}, p_{ij}\in \bbP_{n_{ij}},~0\not\equiv q\in\bbP_{d}\right\},
\end{equation}
$n_{ij}\ge 0$ and $d\ge 0$ are prescribed integers. The $(i,j)$th entry function $\frac{p_{ij}(x)}{q(x)}$ of $R(x)$ is said to be a rational function of type $(n_{ij},d)$ where $p_{ij}\in \bbP_{n_{ij}}$ is the $(i,j)$th entry of a matrix-valued polynomial $P\in \scrP_{(s,t)}:\bbC\rightarrow \bbC^{s\times t}$.

As in the classical scalar case \cite{zhan:2026}, the underlying set $\mathscr{B}$ on which \eqref{eq:bestf0} is posed plays a decisive role. To present the continuum and discrete cases in a unified way, we use the following setting for \eqref{eq:bestf0}.
\begin{itemize}
\item If not specified otherwise, $\mathscr{B} \subset \mathbb{C}$ represents a general compact set; it can be a continuum or a finite set of points. 

\item A matrix-valued function $F(x)$ in $\mathfrak{C}(\mathscr{B})$ means that each entry $f_{ij}(x)$ is continuous on $\mathscr{B}$.

\item $\mathscr{D} \subset \mathbb{C}$ denotes a   continuum with boundary $\mathscr{L} := \partial \mathscr{D}$. Typical examples include closed intervals in the real or imaginary axis, the unit circle, and simply connected domains bounded by Jordan curves. A matrix-valued function $F(x)$ in $\mathfrak{C}_{A}(\mathscr{D})$ means that each entry $f_{ij}(x)$ is continuous on $\mathscr{D}$ and analytic in its interior whenever $\mathrm{int}(\mathscr{D}) \neq \emptyset$.

\item $\mathscr{X} = \{x_j\}_{j=1}^m$  is a  discrete set. In particular, as in \cite{zhzz:2025}, we assume
\begin{equation}\label{eq:numberm}
m\ge  \max_{1\le i\le s,1\le j\le t }\{n_{ij}+d+2\}.
\end{equation} 
For an $F\in \mathfrak{C}_{A}(\mathscr{D})$, the maximum Frobenius norm principle  (see  \cite[Theorem 1]{cond:2020}) implies that the Frobenius norm $\|E(x)\|_{\F}$ of the matrix-valued error function 
\begin{equation}\label{eq:Errfun}
E(x)=F(x) - R(x): \mathscr{D} \rightarrow \bbC^{s\times t}
\end{equation} 
attains its maximum on the boundary $\mathscr{L}=\partial \mathscr{D}$ 
for any $R(x)\in\scrR_{(s,t)}$ with poles outside $\mathscr{D}$; that is, 
$$\|F(x) - R(x)\|_{\infty, \mathscr{D}}:=\sup_{x \in \mathscr{D}} \|F(x) - R(x)\|_{\text{F}}^2 =\max_{x \in \partial\mathscr{D}} \|F(x) - R(x)\|_{\text{F}}^2.$$
Thus, numerical methods for finding a minimax approximation $R(x)$ of $F(x)$ on $\mathscr{D}$ can be designed to solve \eqref{eq:bestf0} with carefully chosen boundary nodes $\mathscr{B}=\mathscr{X} \subset \partial \mathscr{D}$. 
\end{itemize} 

Problem \eqref{eq:bestf0} therefore can describe  both a continuum approximation problem and, in computational practice, a discrete rational minimax problem. Passing from $\mathscr{D}$ to $\mathscr{L}$ or to a sampled set $\mathscr{X}\subseteq\mathscr{L}$ is natural because of the maximum Frobenius norm principle, but it is not automatic that the minimax approximants on these sets coincide. Even in the scalar-valued complex case, rational minimax approximants may depend strongly on the domain; examples in \cite{this:1993,regu:1983,zhan:2026} show that the minimax approximant on a planar domain can differ from the one on its boundary. Hence a discrete approximation method must be supported by conditions under which the computed discrete solution is also meaningful for the underlying continuum problem. A numerical method for solving \eqref{eq:bestf0} on finite point sets has been developed in \cite{zhzz:2025} and recently applied to nonlinear eigenvalue problems in \cite{zhgz:2026}.

The theoretical status of \eqref{eq:bestf0} is also quite different from that of linear  (or polynomial)  minimax approximation. In the scalar-valued case, optimality conditions for local or global rational minimax approximants are less complete than the corresponding theory for linear minimax approximations \cite{chen:1982,mein:1967,rice:1969,rish:1961,shap:1971,sing:1970}. Rational minimax approximation exhibits substantial domain dependence, and no universal necessary and sufficient condition is available for global minimax solutions on arbitrary sets \cite{gutk:1983,sing:2006}. Further distinctions arise between continuum domains and discrete point sets, and between real-valued and complex-valued approximation. The complex scalar-valued theory up to 1983 is surveyed in \cite{gutk:1983}; recent results for the scalar-valued setting, including boundary and discrete formulations, are developed in \cite{zhan:2026}. The matrix-valued problem \eqref{eq:bestf0}, with a Frobenius norm error and a common denominator shared by all entries, adds another layer of structure and appears to be a new problem requiring both approximation-theoretic foundations and practical numerical algorithms.

{\bf Contributions}. 
The purpose of this paper is to develop such a foundation for \eqref{eq:bestf0}. Our contributions are as follows.
\begin{enumerate}
\item[1)] For $F\in \mathfrak{C}(\mathscr{B})$ where $\mathscr{B}$ is dense in itself (i.e., it has no isolated points),  we establish existence of a matrix-valued rational minimax approximant \eqref{eq:bestf0}. This extends Walsh's foundational existence theorem for scalar rational approximation \cite{wals:1931,wals:1969} to matrix-valued functions with a common denominator.
\item[2)]   We derive optimality conditions for matrix-valued rational minimax approximation. In particular, we extend Kolmogorov-type criteria \cite{kolm:1948a} and Ruttan's sufficient global optimality condition \cite{rutt:1985,this:1993} to the matrix-valued setting.
\item[3)]  We analyze the relation among Kolmogorov's criteria, Ruttan's conditions, a dual-based framework \cite{zhzz:2025}, and the {\sf m-d-Lawson} iteration \cite{zhzz:2025}. For discrete approximation, we show how strong duality connects Ruttan's sufficient global optimality to the dual formulation introduced in \cite{zhzz:2025}, and how the optimality equations solved in the {\sf m-d-Lawson} method can be interpreted as Kolmogorov dual criteria.
\item[4)]  We clarify the continuum-to-discrete connection. For $F\in\mathfrak{C}_{A}(\mathscr{D})$, we use the maximum Frobenius norm principle and Ruttan's sufficient condition to identify when the minimax approximant on $\mathscr{D}$ is also characterized through the boundary $\mathscr{L}$ or through a finite set $\mathscr{X}\subseteq\mathscr{L}$ containing suitable extreme points.
\end{enumerate}

The rest of this paper is organized as follows. Section \ref{sec:existence} proves the existence result. The next two sections develop Kolmogorov (Section \ref{sec:kolmogorov}) and Ruttan optimality conditions  (Section \ref{sec:ruttan})  and discuss their implications for continuum, boundary, and discrete problems. Section \ref{sec:algorithm} studies the dual formulation and the {\sf m-d-Lawson} iteration for discrete rational minimax approximation \cite{zhzz:2025}, including their links to Ruttan's condition and Kolmogorov's criterion. Finally, concluding remarks are given in Section \ref{sec:conclude}.

{\bf Notation}.  The notation system of this paper primarily follows that of \cite{zhzz:2025}. We use $\mathbb{C}^{n \times m}$ ($\mathbb{R}^{n \times m}$) for complex (real) matrices, with $\bullet^{\HH}$ ($\bullet^{\T}$) denoting conjugate transpose (transpose). Scalars are Latin ($a$) or Greek ($\alpha$) letters; vectors are bold lowercase ($\ba$, $\bm{\alpha}$), and matrices are uppercase ($A$, $\bm{\Phi}$). Subscripts ($\bp_{ij}$, $F_{ij}$) further denote vectors/matrices, expressed as $\ba = [a_1, \dots, a_n]^{\T}$ or $A = [a_{ij}]$. The calligraphic symbol $\mathscr{B}$ refers to a set.
The identity matrix is $I_n = [\be_1, \dots, \be_n] \in \mathbb{R}^{n \times n}$, with $\be_i$ its $i$-th column. The all-ones vector is $\mathbf{1}_g \in \mathbb{R}^g$. Complex $\mu$ splits as $\mu = \mu^{\tt r} + {\tt i} \mu^{\tt i}$, where ${\tt i} = \sqrt{-1}$, with conjugate $\bar{\mu}$, real part $\mu^{\tt r}$, and modulus $|\mu|$. 
Diagonal matrices are $\diag(\bx) = \diag(x_1, \dots, x_n)$, and element-wise division is $\bx ./ \by = [x_1/y_1, \dots, x_n/y_n]$. The span of $A$ is ${\rm span}(A)$ and $\langle A,B\rangle =\mathrm{tr}(A^{\HH}B)$ with $\|A\|_{\F}=\sqrt{\langle A,A\rangle}$.


\section{The existence}\label{sec:existence}
For a scalar-valued continuous function $f\in \mathfrak{C}(\mathscr{B})$, i.e., $s=t=1$, Walsh's foundational result \cite{wals:1931,wals:1969} establishes the existence  
of a rational minimax approximant of \eqref{eq:bestf0} when $\mathscr{B}$  is dense in   itself (i.e., it has no isolated points). In this section, we shall extend this existence to general  $s$ and $t$. 
\begin{theorem}\label{thm:existence}
Let the matrix-valued continuous function $F(x)\in \mathfrak{C}(\mathscr{B})$  in \eqref{eq:Fx} be defined  on a point set\footnote{For the existence of $\hat R$ of \eqref{eq:bestf0}, the set $\mathscr{B}$ does not need to be closed.}  $\mathscr{B}\subseteq \bbC$ which is dense in   itself. Suppose there exists at least one rational function $R\in \scrR_{(s, t)}$ such that  $\|R(x)-F(x)\|_{\F}$ is bounded on $\mathscr{B}$.  Then the infimum $\eta_{\infty,\mathscr{B}}$ of  \eqref{eq:bestf0} is attainable by a  minimax approximant $\hat R\in \scrR_{(s,t)}$. 
\end{theorem}
\begin{proof}
The proof follows the approach of \cite[Theorem III]{wals:1931}, with an additional argument to preserve the shared denominator. We present the argument for the general matrix-valued case, since the only extra step beyond the scalar proof is a finite diagonal extraction over the \(st\) entries.

First, as there is at least one rational function \(R\in \scrR_{(s,t)}\) such that \(\|R(x)-F(x)\|_{\F}\) is bounded on \(\mathscr{B}\), the infimum \(\eta_{\infty,\mathscr{B}}<\infty\). Hence there is a sequence
\[
R^{(k)}=\frac{P^{(k)}}{q^{(k)}}=[r_{ij}^{(k)}]_{i,j}\in \scrR_{(s,t)},\qquad
r_{ij}^{(k)}=\frac{p_{ij}^{(k)}}{q^{(k)}},
\]
where \(p_{ij}^{(k)}\in\mathbb P_{n_{ij}}\) and \(0\not\equiv q^{(k)}\in\mathbb P_d\). {Multiplying \(P^{(k)}\) and \(q^{(k)}\) by the same nonzero constant if necessary, we normalize each denominator \(q^{(k)}\) so that its leading nonzero coefficient is one; this does not change \(R^{(k)}\) or its approximation error.} The sequence is chosen so that
\begin{equation}\label{eq:seqinf}
\lim_{k\rightarrow \infty}\sup_{x\in \mathscr{B}}\|F(x)-R^{(k)}(x)\|_{\F}^2= \eta_{\infty,\mathscr{B}}.
\end{equation}
For each fixed entry \((i,j)\), the sequence \(\{r_{ij}^{(k)}\}\) is uniformly bounded at any finite set of \(2\max\{n_{ij},d\}+1\) points at which the errors in \eqref{eq:seqinf} are bounded. Hence, by Walsh's   theorem for rational functions \cite[Theorem II]{wals:1931}, and by applying a finite diagonal extraction over all entries, there is a subsequence, still indexed by \(k\), such that every \(r_{ij}^{(k)}\) converges continuously\footnote{A sequence converges continuously in a region if its convergence is uniform in any closed subregion.} in the extended plane \(\overline{\mathbb C}\), except possibly at most $\max\{n_{ij},d\}$ points, to an irreducible rational function
\[
r_{ij}(x)=\frac{\pi_{ij}(x)}{\sigma_{ij}(x)},\qquad \pi_{ij}\in\mathbb P_{n_{ij}},\quad \sigma_{ij}\in\mathbb P_d,
\]
where $\sigma_{ij}$ is monic. 
It remains to show that these limiting entries form an admissible matrix-valued rational function with one common denominator of degree at most \(d\), without increasing the prescribed numerator degrees, and that this limiting function attains the infimum \(\eta_{\infty,\mathscr B}\).

We recall  Walsh's argument   \cite[Theorem I]{wals:1931}.  {Passing to a further subsequence if necessary, the zeros of the normalized denominators \(q^{(k)}\), counted with multiplicity, have finite limit points \(\alpha_1,\ldots,\alpha_\ell\). Let \(\varrho_u\) be the number of zeros of \(q^{(k)}\), counted with multiplicity, that tend to \(\alpha_u\).} Then
$
\sum_{u=1}^{\ell}\varrho_u\le d.
$
Set
\[
\hat q(x)=\prod_{u=1}^{\ell}(x-\alpha_u)^{\varrho_u}.
\]
Indeed, this is a polynomial of degree \(\sum_{u=1}^{\ell}\varrho_u\le d\), with the convention that the empty product is \(1\).
{Since the poles of each \(r_{ij}^{(k)}\) are among the zeros of the common denominator \(q^{(k)}\), Walsh's argument for \cite[Theorem I]{wals:1931} implies that every pole of the irreducible limit \(r_{ij}=\pi_{ij}/\sigma_{ij}\) belongs to \(\{\alpha_1,\ldots,\alpha_\ell\}\).} For a fixed entry \((i,j)\), let \(\mu_{u,ij}\) be the multiplicity of \(\alpha_u\) as a pole of \(r_{ij}\); if \(\alpha_u\) is not a pole, set \(\mu_{u,ij}=0\). {By the argument-principle argument of Walsh \cite[p. 675]{wals:1931}, applied on a small circle around \(\alpha_u\), the pole multiplicity of the limit \(r_{ij}\) at \(\alpha_u\) cannot exceed the number \(\varrho_u\) of zeros of \(q^{(k)}\) tending to \(\alpha_u\). Hence \(\mu_{u,ij}\le \varrho_u\).}
 We claim that
\begin{equation}\label{eq:limit_common_num_degree}
\hat p_{ij}(x):=\pi_{ij}(x)\prod_{u=1}^{\ell}(x-\alpha_u)^{\varrho_u-\mu_{u,ij}}\in\mathbb P_{n_{ij}}.
\end{equation}
Indeed, choose a positively oriented circle \(\mathscr C_{ij}\) whose interior contains all finite zeros of \(\hat q\) and all zeros of \(\pi_{ij}\), while the circumference \(\mathscr C_{ij}\) itself contains no zero or pole of \(r_{ij}\) and no exceptional point of the convergence.
Such a circle exists because these zeros, poles, and exceptional points form a finite set.
For all sufficiently large \(k\), the number of zeros of \(q^{(k)}\) inside \(\mathscr C_{ij}\) is \(\sum_{u=1}^{\ell}\varrho_u\).  Since \(r_{ij}\) has no zero or pole on \(\mathscr C_{ij}\), the image curve \(r_{ij}(\mathscr C_{ij})\) has positive distance from the origin. The continuous convergence of \(r_{ij}^{(k)}\) to \(r_{ij}\) on \(\mathscr C_{ij}\) therefore implies that, for all sufficiently large \(k\), \(r_{ij}^{(k)}\) is also nonzero on \(\mathscr C_{ij}\) and the curves \(r_{ij}^{(k)}(\mathscr C_{ij})\) and \(r_{ij}(\mathscr C_{ij})\) are homotopic in \(\mathbb C\setminus\{0\}\). Hence their winding numbers about the origin are equal. By the argument principle, 
\[
{
{\rm zer}_{\mathscr C_{ij}}\!\left(p_{ij}^{(k)}\right)-\sum_{u=1}^{\ell}\varrho_u
=\deg(\pi_{ij})-\sum_{u=1}^{\ell}\mu_{u,ij},
}
\]
{where \({\rm zer}_{\mathscr C_{ij}}(p_{ij}^{(k)})\) denotes the number of zeros of \(p_{ij}^{(k)}\) inside \(\mathscr C_{ij}\), counted with multiplicity. Therefore}
\[
{
\deg(\pi_{ij})+\sum_{u=1}^{\ell}(\varrho_u-\mu_{u,ij})
={\rm zer}_{\mathscr C_{ij}}\!\left(p_{ij}^{(k)}\right)
\le \deg(p_{ij}^{(k)})\le n_{ij}.
}
\]
{This proves \eqref{eq:limit_common_num_degree}.}

Consequently,
\(
r_{ij}(x)=\frac{\hat p_{ij}(x)}{\hat q(x)},  \hat p_{ij}\in\mathbb P_{n_{ij}},
\)
for every \(1\le i\le s\), \(1\le j\le t\). Therefore
\(
\hat R(x):=\left[\frac{\hat p_{ij}(x)}{\hat q(x)}\right]_{i,j}\in \scrR_{(s,t)}.
\)

It remains to prove that \(\hat R\) attains the infimum. Following the last step of Walsh's proof \cite[p. 674]{wals:1931}, suppose, to the contrary, that
$
\sup_{x\in \mathscr{B}}\|F(x)-\hat R(x)\|_{\F}^2>\eta_{\infty,\mathscr{B}}.
$
Then there are \(\epsilon>0\) and \(x_0\in\mathscr B\) such that either \(\hat R\) has a nonremovable pole at \(x_0\), or, after removable continuation if \(x_0\) is a removable zero of \(\hat q\),
\[
\|F(x_0)-\hat R(x_0)\|_{\F}^2>\eta_{\infty,\mathscr B}+2\epsilon.
\]
Let \(\mathscr S\) be the finite set consisting of the exceptional points where the entrywise convergence may fail and the zeros of \(\hat q\). If \(x_0\) is not a nonremovable pole of \(\hat R\), then the continuity of \(F-\hat R\) at \(x_0\) gives a neighborhood \(\mathscr{U}\) of \(x_0\) such that
\(\|F(x)-\hat R(x)\|_{\F}^2>\eta_{\infty,\mathscr B}+\epsilon\)
for all \(x\in \mathscr{U}\cap\mathscr B\). If \(x_0\) is a nonremovable pole, the same strict inequality holds for all \(x\in (\mathscr{U}\setminus\{x_0\})\cap\mathscr B\) after taking \(\mathscr{U}\) sufficiently small, because \(\|F(x)-\hat R(x)\|_{\F}\to\infty\) as \(x\to x_0\). Since \(\mathscr B\) is dense in itself, this set contains infinitely many points of \(\mathscr B\); removing the finite set \(\mathscr S\), we can choose \(x_1\in\mathscr B\setminus\mathscr S\) such that
$
\|F(x_1)-\hat R(x_1)\|_{\F}^2>\eta_{\infty,\mathscr B}+\epsilon,
$
and the subsequence \(R^{(k)}(x_1)\) converges to \(\hat R(x_1)\). Hence, for all sufficiently large \(k\),
\[
\|F(x_1)-R^{(k)}(x_1)\|_{\F}^2>\eta_{\infty,\mathscr B}+\frac{\epsilon}{2} \Longrightarrow
\sup_{x\in\mathscr B}\|F(x)-R^{(k)}(x)\|_{\F}^2>\eta_{\infty,\mathscr B}+\frac{\epsilon}{2}
\]
for infinitely many \(k\). This contradicts \eqref{eq:seqinf}. Therefore \(\hat R\) is a minimax approximant, and the proof is complete.
\end{proof}

 \section{Kolmogorov criteria in primal and dual forms}\label{sec:kolmogorov}
Kolmogorov \cite{kolm:1948a} established the optimality condition for the polynomial minimax approximant of a continuous scalar-valued function by a variational argument resulting from local perturbation on the best approximant. The Kolmogorov conditions for the rational scalar-valued case were later   discussed in \cite[Theorem 3]{will:1979} (see also \cite[Theorem 1.2]{rutt:1985} and \cite{zhan:2026}). For the matrix-valued case in \eqref{eq:bestf0} for  $F:\mathbb{C}\to\mathbb{C}^{s\times t}$, we next extend the Kolmogorov criteria  to  characterize the local minimizer $R(x)=P(x)/q(x)\in\mathscr{R}_{(s,t)}$ of \eqref{eq:bestf0}. 

\subsection{The irreducible representation of $R(x)\in \mathscr{R}_{(s,t)}$}\label{subsec:irreducible}
As a reducible rational function may admit various representations, we should first specify a unique representation  to describe the associated optimality of a local minimizer. For this purpose, in the rest of our paper, for a given \({R}(x)= {P}(x)/ {q}(x)\in \mathscr{R}_{(s,t)}\), we assume its unique representation of \( {R(x)}\) obtained by canceling all common factors of the numerators \(\{ {p}_{ij}(x)\}\)  in $P(x)$ and denominator \( {q}(x)\); that is, the greatest common divisor of ${p}_{11}(x),\dots, {p}_{st}(x)$ and ${q}(x)$ is $1$. Thus, the representation \( {R}(x)= {P}(x)/ {q}(x)\)  satisfies  \[{\rm gcd}(P,q):={\rm gcd}\left( {p}_{11}(x),\dots, {p}_{st}(x), {q}(x)\right)=1;\] 
moreover, we define its {\it defect} as 
\begin{equation}\label{eq:tilde_nu_def}
   {\nu}(P,q):=\min\left\{\min_{1 \le i \le s,1 \le j \le t}\left\{n_{ij}-{\rm deg}\left( {p}_{ij}\right)\right\},d-{\rm deg}\left( {q}\right)\right\}.
\end{equation}

\begin{example}
Suppose \( {R}=\frac{1}{ {q}}[ {p}_{11}, {p}_{21}]^{\rm T}=\frac{1}{(x-1)^2(x-2)}[(x-1)^2x,(x-1)(x-2)]^{\rm T} \in \mathscr{R}_{(2,1)}\) with \(d=n_{11}=n_{21}=3\) in \eqref{eq:rats}. By canceling all common factors of \( {p}_{11}, {p}_{21}\) and \( {q}\), we obtain its unique representation 
  \begin{equation}\nonumber
     {R}=\frac{1}{ {q}}\begin{bmatrix}
       {p}_{11}\\  {p}_{21}
    \end{bmatrix}=\frac{1}{(x-1)(x-2)}\begin{bmatrix}(x-1)x \\ x-2 \end{bmatrix},
  \end{equation}
  with \( {\nu}(P,q)=1\) in \eqref{eq:tilde_nu_def}.
\end{example}

\subsection{The Kolmogorov primal criterion}
Suppose now \(\hat{R}=\hat{P}/\hat{q}=[\hat{p}_{ij}/\hat{q}] \in \mathscr{R}_{(s,t)}\) is a local minimizer of \eqref{eq:bestf0} for $F\in \mathfrak{C}(\mathscr{B})$ on the compact set $\mathscr{B}\subset \bbC$. It is easy to see that $\hat R$ is pole-free in $\mathscr{B}$ because otherwise the maximum approximation error 
\begin{equation}\label{eq:norm_B}
  \|F-\hat R\|_{\infty, \mathscr{B}}:=\sup_{x \in \mathscr{B}}\|F(x)-\hat R(x)\|_{\F}^2
\end{equation}
is infinite.  The following theorem describes the primal form of the Kolmogorov criterion, which generalizes the Kolmogorov condition for the scalar-valued functions \cite[Theorem 2.1]{rutt:1985}.

\begin{theorem}[Kolmogorov primal criterion for matrix-valued functions]\label{thm:kolmogorov_primal}
Let  $\mathscr{B}$ be a compact set of $\bbC$ and $\hat R=\hat P/\hat q\in\mathscr{R}_{(s,t)}$ be a local minimizer of \eqref{eq:bestf0} with \({\rm gcd}\left(\hat P,\hat{q}\right)=1\) for $F\in \mathfrak{C}(\mathscr{B})$. 
Then 
\begin{equation}\label{eq:matrixKolmogorov}
\max_{x\in\mathcal{E}(\hat R)}\mathrm{Re} \left\langle \hat E(x),h(x)\hat{P}(x)-\hat{q}(x)G(x)\right\rangle  \geq 0,\ \forall (G,h) \in \mathscr{P}_{(s,t)}\times\mathbb{P}_d,
\end{equation}
where 
\begin{equation}\label{eq:Ehat_def}
  \hat E(x)=\frac{(F(x)-\hat R(x))\hat q(x)}{  \overline{\hat q(x)}},
\end{equation}
$\mathcal{E}(\hat R)=\left\{x:\|F(x)-\hat R (x)\|_{\F}^2=\|F-\hat{R}\|_{\infty,\mathscr{B}},~x\in \mathscr{B} \right\}$ and $\langle A,B\rangle =\mathrm{tr}(A^{\HH}B)$.
\end{theorem}
\begin{proof}
  Following Ruttan's proof of \cite[Theorem 1.2]{rutt:1985}, we proceed by contradiction. Assume that there exists \((G,h)\in \mathscr{P}_{(s,t)}\times \mathbb{P}_{d}\) satisfying (note that $\mathcal{E}(\hat R)$ is compact)
  \begin{equation}\nonumber
    {\rm Re}\left\langle\hat{E}(x),h(x)\hat{P}(x)-\hat{q}(x)G(x)\right\rangle<0,\ \forall x \in {\cal E}(\hat{R}).
  \end{equation}
  Consider the difference \(\delta_{\lambda}(x)\) for sufficient small \(\lambda \in \mathbb{R}\) in the form:
  \begin{equation}\label{eq:difference_rat_kol}
    \delta_{\lambda}(x) := \|F(x)-R_{\lambda}(x)\|_{\rm F}^2 -\|F(x)-\hat{R}(x)\|_{\rm F}^2 = \frac{2\lambda \kappa(x)+\lambda^2\rho(x) }{|q_{\lambda}(x)|^2},
  \end{equation}
  where \(R_{\lambda}(x):=P_{\lambda}(x)/q_{\lambda}(x)\), \(P_{\lambda}(x):=\hat{P}(x)+\lambda G(x)\), \(q_{\lambda}(x):=\hat{q}(x)+\lambda h(x)\) and 
  \begin{equation}\label{eq:kappa_def}
    \kappa(x)={\rm Re}\left\langle\hat{E}(x),h(x)\hat{P}(x)-\hat{q}(x)G(x)\right\rangle,
  \end{equation}
  \begin{equation}\label{eq:rho_def}
    \rho(x)=\|F(x)h(x)-G(x)\|_{\rm F}^2-\|F(x)h(x)-\hat{R}(x)h(x)\|_{\rm F}^2.
  \end{equation}
  
With the compactness of \({\cal E}(\hat{R})\), there exist an open set \(\mathscr{U}\) satisfying \({\cal E}(\hat{R}) \subset \mathscr{U} \subset \mathscr{B}\) and an \(\epsilon >0\) such that \(\kappa(x)<-\epsilon\) on \(\mathscr{U}\). Since the first-order term \(\kappa(x)\) dominates \(\delta_{\lambda}(x)\) and \(q_{\lambda}(x)\) is nonzero on \(\mathscr{B}\) as $\lambda\rightarrow 0$, for any sufficiently small \(\lambda \in \mathbb{R}\),  the right side of \eqref{eq:difference_rat_kol} is negative on \(\mathscr{U}\), leading to 
$$\|F(x)-R_{\lambda}(x)\|_{\rm F}^2<\|F-\hat{R}\|_{\infty,\mathscr{B}},   ~\forall x \in \mathscr{U}.$$ On the other hand,   there exists a \(\zeta>0\) such that $$\forall x \in \mathscr{B}\setminus \mathscr{U},~ \|F(x)-\hat{R}(x)\|_{\rm F}^2<\|F-\hat{R}\|_{\infty,\mathscr{B}} -\zeta.$$ Therefore, from \eqref{eq:difference_rat_kol}, \(\|F(x)-R_{\lambda}(x)\|_{\rm F}^2<\|F-\hat{R}\|_{\infty,\mathscr{B}} -\zeta+O(\lambda)\), leading to \(\|F(x)-R_{\lambda}(x)\|_{\rm F}^2<\|F-\hat{R}\|_{\infty,\mathscr{B}}\) for \(x \in \mathscr{B}\setminus \mathscr{U}\) and all \(\lambda>0\) sufficiently small. This implies that $R_{\lambda}(x)$ locally achieves the smaller $\|F-{R}_\lambda\|_{\infty,\mathscr{B}}$ than $\|F-\hat{R}\|_{\infty,\mathscr{B}}$, a contradiction.  
\end{proof}
 
Since polynomial minimax approximation  for matrix-valued functions is a special case of rational approximation with \(d=0\) and \(\hat{q}\equiv1\), the matrix-valued Kolmogorov criteria for polynomial minimax approximation can be obtained from Theorem \ref{thm:kolmogorov_primal} straightforwardly.
\begin{corollary}\label{coro:poly_kolmogorov}
  Let \(\mathscr{B}\) be a compact set of \(\mathbb{C}\) and $$\hat{P} \in \arg \min_{P \in \mathscr{P}_{(s,t)}}\|F-P\|_{\infty,\mathscr{B}},$$ i.e. \(\hat{P}\) is a matrix-valued polynomial (local) minimax approximation to \(F \in \mathfrak{C}(\mathscr{B})\). Then 
  \begin{equation}\label{eq:poly_kol}
    \max_{z \in {\cal E}(\hat{P})} {\rm Re}\left\langle F(z)-\hat{P}(z),G(z)\right\rangle\ge 0,\ \ \forall G \in \mathscr{P}_{(s,t)},
  \end{equation}
  where \({\cal E}(\hat{P})=\left\{x:\|F(x)-\hat{P}(x)\|_{\rm F}^2=\|F-\hat{P}\|_{\infty,\mathscr{B}}, x\in \mathscr{B}\right\}\).
\end{corollary}

\subsection{The Kolmogorov dual criterion}
In this section, we shall rely on Carath\'eodory lemma and Hahn-Banach separation theorem to establish the dual form of  the Kolmogorov condition. 
\begin{lemma}\label{lem:0_in_convA}
  Suppose \(\mathscr{B} \subset \mathbb{C}\) is a compact set and \(\hat{R}=\hat P/\hat q=[\hat{p}_{ij}/\hat{q}] \in \mathscr{R}_{(s,t)}\) is a local minimizer of \eqref{eq:bestf0} with \({\rm gcd}(\hat P,\hat{q})=1\). Then \(\pmb{0} \in {\rm conv}\left({\cal A}\right)\), where
  \begin{equation}\label{eq:calA_def}
    {\cal A}=\left\{\begin{bmatrix}
      \overline{\tau_{11}(x)}\hat{q}(x)\boldsymbol{\psi}_{11}(x) \\ \vdots \\ \overline{\tau_{st}(x)}\hat{q}(x)\boldsymbol{\psi}_{st}(x) \\ \sum\limits_{i=1}^{s}\sum\limits_{j=1}^{t}\overline{\tau_{ij}(x)}\hat{p}_{ij}(x)\boldsymbol{\phi}(x)
    \end{bmatrix}\in \mathbb{C}^{ \widetilde{n}}:x \in {\cal E}(\hat{R})\right\},
  \end{equation}
  \begin{equation}\label{eq:widetilde_N_def}
    \widetilde{n}=\sum\limits_{i=1}^{s}\sum\limits_{j=1}^{t}\left(n_{ij}+1\right)+d+1,
  \end{equation}
  \begin{equation}\label{eq:t_ij_def}
    \tau_{ij}(x)=\frac{f_{ij}(x)\hat{q}(x)-\hat{p}_{ij}(x)}{\overline{\hat{q}(x)}},\ \ 1 \le i \le s,1 \le j \le t,
  \end{equation}
  \begin{equation}\label{eq:Phi_Psi_def}
    \boldsymbol{\phi}(x)=[\phi_0(x),\dots,\phi_d(x)]^{\rm T},\ \ \boldsymbol{\psi}_{ij}(x)=[\psi_0(x),\dots,\psi_{n_{ij}}(x)]^{\rm T},\ \ 1 \le i \le s,1 \le j \le t.
  \end{equation}
\end{lemma}
\begin{proof}
We prove this lemma by contradiction. Assume that \(\pmb{0} \not\in {\rm conv}\left({\cal A}\right)\), then by Hahn-Banach separation theorem, there exist a \(\gamma \in \mathbb{R}\) and \(\pmb{c}=[\pmb{a}_{11}^{\rm T},\dots,\pmb{a}_{st}^{\rm T},\pmb{b}^{\rm T}]^{\rm T} \in \mathbb{C}^{ \widetilde{n}}\) with \(\pmb{a}_{ij} \in \mathbb{C}^{n_{ij}+1}\), \(\pmb{b} \in \mathbb{C}^{d+1}\) \(\forall 1 \le i \le s,\ 1 \le j \le t\), such that \({\rm Re}\left(\pmb{c}^{\rm H}\pmb{\xi}\right)<\gamma<{\rm Re}\left(\pmb{c}^{\rm H}\pmb{0}\right)=0\) for each \(\pmb{\xi} \in {\rm conv}\left({\cal A}\right)\). Then we have
  \begin{equation}\nonumber
    {\rm Re}\left(\sum\limits_{i=1}^{s}\sum\limits_{j=1}^{t}\overline{\tau_{ij}(x)}\left(\hat{q}(x)\pmb{a}_{ij}^{\rm H}\boldsymbol{\psi}_{ij}(x)+\hat{p}_{ij}(x)\pmb{b}^{\rm H}\boldsymbol{\phi}(x)\right)\right)<0,\ \ \forall x \in {\cal E}(\hat{R}).
  \end{equation}
Letting \(G(x)=[\pmb{a}_{ij}^{\rm H}\boldsymbol{\psi}_{ij}(x)] \in \mathscr{P}_{(s,t)}\) and \(h(x)=-\pmb{b}^{\HH}\boldsymbol{\phi}(x) \in \mathbb{P}_{d}\), we then construct a pair \((G,h) \in \mathscr{P}_{(s,t)}\times \mathbb{P}_d\) which contradicts the primal form of Kolmogorov's condition \eqref{eq:matrixKolmogorov}.
\end{proof}

\begin{lemma}\label{lem:dim_span_A}
  Under the assumption of Lemma \ref{lem:0_in_convA},  the set ${\cal A}$ in \eqref{eq:calA_def} is contained in a subspace of $\bbC^{\wtd  n}$ with dimension no larger than
  \begin{equation}\label{eq:dim_spanA}
 \mathring{n}:=\sum\limits_{i=1}^s\sum\limits_{j=1}^t\left(n_{ij}+1\right)+d-{\nu}(\hat P, \hat q).
  \end{equation}
\end{lemma}
\begin{proof}
  By the definition of \(\nu={\nu}(\hat P, \hat q)\) in \eqref{eq:tilde_nu_def}, we know that \(\xi\hat{q} \in \mathbb{P}_d\) and \(\xi\hat{p}_{ij} \in \mathbb{P}_{n_{ij}}\) \(1 \le i \le s,\ 1 \le j \le t\) for each \(\xi \in \mathbb{P}_{ {\nu}}\). Fixing  a basis \(\mathbb{P}_{n_{ij}}={\rm span}\left(\psi_0(x),\dots,\psi_{n_{ij}}(x)\right)\), it is known that \(\xi\hat{p}_{ij}\) has a unique coordinate vector \(\ba_{ij} \in \mathbb{C}^{n_{ij}+1}\) with respect to this basis satisfying \(\xi\hat{p}_{ij}(x)=[\psi_0(x),\dots,\psi_{n_{ij}}(x)]\pmb{a}_{ij}\) for which we denote by \(\ba_{ij}=\bt\left(\xi\hat{p}_{ij},\mathbb{P}_{n_{ij}}\right)\). Similarly, we denote by \(\bt\left(\xi\hat{q},\mathbb{P}_d\right)\) the coordinate vector of \(\xi\hat{q}\) with respect to \(\left\{\phi_0(x),\dots,\phi_d(x)\right\}\).

With the linear mapping
  \begin{equation}\label{eq:Pv_C_iso_def}
    \textswab{T}:\xi \in \mathbb{P}_{ {\nu}} \longmapsto \begin{bmatrix}
      \overline{\bt\left(\xi\hat{p}_{11},\mathbb{P}_{n_{11}}\right)}\\ \vdots \\ \overline{\bt\left(\xi\hat{p}_{st},\mathbb{P}_{n_{st}}\right)} \\ -\overline{\bt\left(\xi\hat{q},\mathbb{P}_{d}\right)}
    \end{bmatrix} =:\textswab{T}(\xi)\in  \mathbb{C}^{ \widetilde{n}},
  \end{equation}
we have that \({\rm dim}\left(\textswab{T}\right)= {\nu}+1\). Therefore, we obtain the desired result \eqref{eq:dim_spanA} by directly verifying that \({\cal A}   \subseteq  \left(\textswab{T}\left(\mathbb{P}_{\nu}\right)\right)^{\bot}\).
\end{proof}

\begin{lemma}[Carath\'eodory]\label{lem:caratheodory}
    If \({\cal F} \subset \mathbb{R}^n\), then every point of the convex hull \({\rm conv}\left({\cal F}\right)=\left\{\sum\limits_{j}\omega_j\pmb{y}_j:\sum\limits_{j}\omega_j=1,\ \omega_j \ge 0,\ \pmb{y}_j \in {\cal F}\right\}\) can be written as a convex linear combination of at most \(n+1\) points of \({\cal F}\).
\end{lemma}
With Lemmas \ref{lem:0_in_convA},   \ref{lem:dim_span_A} and   \ref{lem:caratheodory}, we obtain the dual form of the Kolmogorov condition:
\begin{theorem}[Kolmogorov dual criterion for matrix-valued functions]\label{thm:matrixKolmogorov_dual}
  Suppose \(\mathscr{B} \subset \mathbb{C}\) is compact and \(\hat{R}=\hat P/\hat q=[\hat{p}_{ij}/\hat{q}] \in \mathscr{R}_{(s,t)}\) is a local minimizer of \eqref{eq:bestf0} with \({\rm gcd} (\hat P,\hat{q} )=1\). Let \( \nu\) and \( \widetilde{n}\) be defined as in \eqref{eq:tilde_nu_def} and \eqref{eq:widetilde_N_def},  respectively. Then there exist \(r\ (1 \le r \le 2\left( \widetilde{n}- {\nu}\right)-1)\) extreme points \(z_1,\dots,z_r \in {\cal E}(\hat{R})\) and positive number \(\omega_1,\dots,\omega_r\) with \(\sum\limits_{k=1}^{r}\omega_k=1\) such that
  \begin{equation}\label{eq:matrixKolmogorovDual}
    \sum\limits_{k=1}^{r}\omega_k \left\langle\hat{E}(z_{k}),h(z_k)\hat{P}(z_k)-G(z_k)\hat{q}(z_k)\right\rangle=0,\ \ \forall (G,h) \in \mathscr{P}_{(s,t)}\times \mathbb{P}_d,
  \end{equation}
  where $\hat E(z)$ is given in \eqref{eq:Ehat_def}.
\end{theorem}
\begin{proof}
  With the isomorphism \(\mathbb{C}^{ \widetilde{n}- {\nu}-1}\cong \mathbb{R}^{2\left( \widetilde{n}- {\nu}-1\right)}\) and  Lemma \ref{lem:dim_span_A}, we know that any point in \({\rm conv}\left({\cal A}\right)\) can be expressed as a convex combination of at most \[2 \mathring{n}+1=2\left( \widetilde{n}- {\nu}\right)-1\] points of \({\cal A}\). Therefore \(\pmb{0} \in {\rm conv}\left({\cal A}\right)\) implies there \( 1 \le r \le2 \mathring{n}+1\) extreme points \(\left\{z_k\right\}_{k=1}^{r} \subset {\cal E} (\hat{R} )\) and positive number \(\left\{\omega_k\right\}_{k=1}^r\) with \(\sum\limits_{k=1}^{r}\omega_k=1\) such that
  \begin{equation}\nonumber
    \sum\limits_{k=1}^{r}\omega_k\overline{\tau_{ij}(z_k)}\hat{q}(z_k)\bm{\psi}_{ij}(z_k)=\pmb{0} \in \mathbb{C}^{n_{ij}+1},\ \ 1 \le i \le s,1 \le j \le t,
  \end{equation}
  \begin{equation}\nonumber
    \sum\limits_{k=1}^{r}\omega_k\left(\sum\limits_{i=1}^{s}\sum\limits_{j=1}^{t}\overline{\tau_{ij}(z_k)}\hat{p}_{ij}(z_k)\bm{\phi}(z_k)\right)=\pmb{0} \in \mathbb{C}^{d+1},
  \end{equation}
  where \(\tau_{ij}(x)\) is defined in \eqref{eq:t_ij_def} and \(\hat{E}(x)=[\tau_{ij}(x)]\). Note that for any \(G(x)=[g_{ij}(x)] \in \mathscr{P}_{(s,t)}\) and \(h(x) \in \mathbb{P}_d\), each \(g_{ij}(x)\) and \(h(x)\) can be expressed as \(g_{ij}(x)=\pmb{a}_{ij}^{\rm T}\bm{\psi}_{ij}(x)\) and \(h(x)=\pmb{b}^{\rm T}\bm{\phi}(x)\), respectively, with \(\pmb{a}_{ij} \in \mathbb{C}^{n_{ij}+1}\) for all \(i,j\) and \(\pmb{b} \in \mathbb{C}^{d+1}\). Then we obtain the desired result.
\end{proof}
Theorem \ref{thm:matrixKolmogorov_dual} implies the special case of polynomial minimax approximation where $d = 0$ and $\hat{q} \equiv 1$. We state this explicitly as Corollary \ref{coro:poly_kolmogorov_dual}.
\begin{corollary}\label{coro:poly_kolmogorov_dual}
  Let \(\mathscr{B}\) be a compact set of \(\mathbb{C}\) and $$\hat{P} \in \arg\min_{P \in \mathscr{P}_{(s,t)}}\|F-P\|_{\infty,\mathscr{B}},$$ i.e. \(\hat{P}\) is a matrix-valued polynomial (local) minimax approximation to \(F \in \mathfrak{C}(\mathscr{B})\). Then there exist $1 \le r \le 2\left(\sum\limits_{i=1}^{s}\sum\limits_{j=1}^{t}\left(n_{ij}+1\right)\right)+1 $ extreme points \(z_1,\dots,z_r \in {\cal E}(\hat{P})\) and positive number \(\omega_1,\dots,\omega_r\) with \(\sum\limits_{k=1}^{r}\omega_k=1\) such that
  \begin{equation}\label{eq:poly_kol_dual}
    \sum\limits_{k=1}^{r}\omega_k\left\langle F(z_{k})-\hat{P}(z_k),G(z_k)\right\rangle=0,\ \ \forall G \in \mathscr{P}_{(s,t)}.
  \end{equation}
\end{corollary}

\section{Ruttan's sufficient optimality for global approximants}\label{sec:ruttan}
In \cite{rutt:1985}, Ruttan established a sufficient condition to characterize global minimax approximants of continuous scalar-valued functions by proving the positive semidefiniteness of a carefully constructed matrix \cite[Theorem 2.1]{rutt:1985}. In this section, we generalize Ruttan's sufficient condition to the matrix-valued setting to identify the global minimizer $\hat R(x) = \hat P(x)/\hat q(x) \in \mathscr{R}_{(s,t)}$ for a continuous function $F: \mathscr{B}\subseteq\mathbb{C} \rightarrow \mathbb{C}^{s \times t}$ in \eqref{eq:bestf0}.

\subsection{Sufficient optimality for global approximants}
Given \(\hat{R}=\hat{P}/\hat{q}=[\hat{p}_{ij}/\hat{q}] \in \mathscr{R}_{(s,t)}\) with \({\rm gcd}(\hat P,\hat{q})=1\) and   \(\|\hat{e}\|_{\infty,\mathscr{B}}=\sup_{z \in \mathscr{B}}\|F(z)-\hat{R}(z)\|_{\rm F}^2\), we introduce the auxiliary function
\begin{equation}\label{eq:auxiliary_func}
  u(x;G,h)=\|F(x)h(x)-G(x)\|_{\rm F}^2-|h(x)|^2\|\hat{e}\|_{\infty,\mathscr{B}},
\end{equation}
defined for a pair \((G,h) \in \mathscr{P}_{(s,t)}\times \mathbb{P}_d\). It is noted that \(u(x;G,h)=\rho(x)\ \forall x \in {\cal E}(\hat{R})\), where \(\rho(x)\) is defined in \eqref{eq:rho_def}. With the polynomial bases \eqref{eq:Phi_Psi_def} of \(\mathbb{P}_d\) and \(\mathbb{P}_{n_{ij}},\ 1 \le i \le s,\ 1 \le j \le t\), letting \(G(x)=[g_{ij}(x)]\) and
\begin{equation}\nonumber
  h(x)=[\phi_0(x),\dots,\phi_d(x)]\pmb{h},\ \pmb{h} \in \mathbb{C}^{d+1},
\end{equation}
\begin{equation}\nonumber
  g_{ij}(x)=[\psi_0(x),\dots,\psi_{n_{ij}}(x)]\pmb{g}_{ij},\ \pmb{g}_{ij} \in \mathbb{C}^{n_{ij}+1},\ 1 \le i \le s,\ 1 \le j \le t,
\end{equation}
it can be directly verified that
\begin{equation}\label{eq:u_equality_with_H}
  u(x;G,h)=\begin{bmatrix}
    \pmb{g} \\ \pmb{h}
  \end{bmatrix}^{\rm H} H(x) \begin{bmatrix}
    \pmb{g} \\ \pmb{h}
  \end{bmatrix},
\end{equation}
where \(\pmb{g}=[\pmb{g}_{11}^{\rm T},\dots,\pmb{g}_{st}^{\rm T}]^{\rm T} \in \mathbb{C}^{n}\), \(n=\sum\limits_{i=1}^s \sum\limits_{j=1}^t \left(n_{ij}+1\right)\) and
\begin{equation}\label{eq:H_x_def}
  H(x)=\begin{bmatrix}
    H_1(x)&H_3(x) \\ H_3^{\rm H}(x)&H_2(x)
  \end{bmatrix} \in \mathbb{C}^{(n+d+1) \times (n+d+1)}
\end{equation} 
is a matrix-valued function with blocks 
\begin{subequations}\label{eq:H1_x_def}
  \begin{equation}
    H_1(x)=\begin{bmatrix}
      H^{(11)}_1(x) && \\ &\ddots& \\ &&H^{(st)}_1(x)
    \end{bmatrix} \in \mathbb{C}^{n\times n},
  \end{equation}
  \begin{equation}
    [H^{(ij)}_1(x)]_{\ell,u}=\overline{\psi_{\ell-1}(x)}\psi_{u-1}(x),\ \ 1 \le \ell,u \le n_{ij}+1,\ 1 \le i \le s,1\le j \le t,
  \end{equation}
\end{subequations}
\begin{equation}\label{eq:H2_x_def}
  [H_2(x)]_{\ell,u}=\left(\|F(x)\|_{\rm F}^2-\|\hat{e}\|_{\infty,\mathscr{B}}\right)\overline{\phi_{\ell-1}(x)}\phi_{u-1}(x) \in \mathbb{C},\ \ 1 \le \ell,u \le d+1,
\end{equation}
\begin{subequations}\label{eq:H3_x_def}
  \begin{equation}
    H_3(x)=\left[ \left(H^{(11)}_3(x)\right)^{\rm T},\dots,\left(H^{(st)}_3(x)\right)^{\rm T} \right]^{\rm T} \in \mathbb{C}^{n\times (d+1)},
  \end{equation}
  \begin{equation}
    [H^{(ij)}_3(x)]_{\ell,u}=-f_{ij}(x)\overline{\psi_{\ell-1}(x)}\phi_{u-1}(x),\ \ 1 \le \ell \le n_{ij}+1,\ 1 \le  u \le d+1,\ \forall i,j.
  \end{equation}
\end{subequations}
One can verify  that if \(\{\widetilde{\phi_k}\}_{k=0}^{d}\) and \(\{\widetilde{\psi_k}\}_{k=0}^{n_{ij}}\)    are new bases for \(\mathbb{P}_d\) and \(\mathbb{P}_{n_{ij}},\ 1 \le i \le s,\ 1 \le j \le t\), with the transformation matrices \(T_{(ij)} \in \mathbb{C}^{(n_{ij}+1) \times (n_{ij}+1)}\) and \(T_2 \in \mathbb{C}^{\left(d+1\right)\times \left(d+1\right)}\) satisfying
\begin{equation}\nonumber
  [\psi_0(x),\dots,\psi_{n_{ij}}(x)]=[\widetilde{\psi}_0(x),\dots,\widetilde{\psi}_{n_{ij}}(x)]T_{(ij)},\ 1 \le i \le s,\ 1 \le j \le t,
\end{equation}
\begin{equation}\nonumber
  [\phi_0(x),\dots,\phi_{d}(x)]=[\widetilde{\phi}_0(x),\dots,\widetilde{\phi}_{d}(x)]T_{2},
\end{equation}
then the new   \(\widetilde{H}(x)\) in \eqref{eq:H_x_def} associated with   \(\{\widetilde{\psi_k}\}_{k=0}^{n_{ij}}\) \(1 \le i \le s\ 1 \le j \le t\) and \(\{\widetilde{\phi_k}\}_{k=0}^{d}\) is 
\begin{equation}\nonumber
  \widetilde{H}(x)=\begin{bmatrix}
    T^{\rm H}_{(11)}&&&\\&\ddots&&\\&&T^{\rm H}_{(st)}&\\&&&T_2^{\rm H}
  \end{bmatrix}H(x)\begin{bmatrix}
    T_{(11)}&&&\\&\ddots&&\\&&T_{(st)}&\\&&&T_2
  \end{bmatrix}.
\end{equation}
This implies that the definiteness of \(H(x)\) is invariant under the choice of bases for \(\mathbb{P}_d\) and \(\mathbb{P}_{n_{ij}}\), \(1 \le i \le s,\ 1 \le j \le t\).

We now present Theorem \ref{thm:ruttan_sufficient},which can be seen as an extension of \cite[Theorem 2.1]{rutt:1985} to the matrix-valued case, for characterizing a global minimax approximant \(\hat{R} \in \mathscr{R}_{(s,t)}\).
\begin{theorem}\label{thm:ruttan_sufficient}
  Suppose \(\mathscr{B} \subset \mathbb{C}\) is a compact set and \(\hat{R}=\hat{P}/\hat{q} \in \mathscr{R}_{(s,t)}\) is an approximant of \(F\in \mathfrak{C}(\mathscr{B})\) with \({\rm gcd}(\hat P,\hat{q})=1\). If there exist extreme points \(\{z_{k}\}_{k=1}^{r} \subset {\cal E}(\hat{R})\), \(r\ge 1\) and positive real constants \(\{\omega_k\}_{k=1}^r\) such that \(\sum\limits_{k=1}^r\omega_k=1\) and the Hermitian matrix
  \begin{equation}\label{eq:ruttan_semipositive}
    \mathbf{H}=\sum\limits_{k=1}^r\omega_kH(z_{k})\succeq 0,
  \end{equation}
  where \(H(x)\) is given by \eqref{eq:H_x_def}. Then \(\hat{R}\) is a global minimax approximant for \eqref{eq:bestf0}. In that case,   \(\{z_k\}_{k=1}^{r}\) and \(\{\omega_k\}_{k=1}^r\) satisfy the local (dual form) Kolmogorov condition \eqref{eq:matrixKolmogorovDual}.
\end{theorem}
\begin{proof}
  We first prove that \(\hat{R}\) is a global minimax approximant of \eqref{eq:bestf0} by contradiction. To see this, suppose, by contrast, that \(\hat{R}\) is not a global solution to \eqref{eq:bestf0}. Then either \(\eta_{\infty,\mathscr{B}}\) is not attainable or is attainable but \(\hat{R}\) is not a solution. In either case, it must  hold that \(\eta_{\infty,\mathscr{B}}<\|\hat{e}\|_{\infty,\mathscr{B}}\), and we can find \(\widetilde{R}=\widetilde{P}/\tilde{q}=[\tilde{p}_{ij}]/\tilde{q}\) with \(\tilde{q}(x) \not= 0\) \((\forall x \in \mathscr{B})\) satisfying
  \begin{equation}\label{eq:important_equality_rutt_theorem_proof}
    \eta_{\infty,\mathscr{B}} \le \|\tilde{e}\|_{\infty,\mathscr{B}}=\max_{z \in \mathscr{B}}\|F(z)-\widetilde{R}(z)\|_{\rm F}^2<\|\hat{e}\|_{\infty,\mathscr{B}}.
  \end{equation}
  {Here the denominator of \(\widetilde R\) may be taken nonzero on \(\mathscr B\): otherwise a nonremovable pole would make the left maximum infinite, while a removable zero can be cancelled from all entries and the common denominator.}
  Let \(\tilde{\pmb{g}}_{ij} \in \mathbb{C}^{n_{ij}+1}\) be the associated coefficient vectors of \(\tilde{p}_{ij}(x)\) in \(\{\psi_k(x)\}_{k=0}^{n_{ij}}\) \(\forall i,j\), and \(\tilde{\pmb{h}} \in \mathbb{C}^{d+1}\) be \(\tilde{q}(x)\)'s in \(\{\phi_k(x)\}_{k=0}^{d}\). Let \(\tilde{\pmb{g}}=[\tilde{\pmb{g}}_{11}^{\rm T},\dots,\tilde{\pmb{g}}_{st}^{\rm T}]^{\rm T}\). By \eqref{eq:ruttan_semipositive},  \eqref{eq:u_equality_with_H} and \eqref{eq:auxiliary_func}, it follows that
  \begin{align}\nonumber
    0 & \le \sum\limits_{k=1}^{r}\omega_k\begin{bmatrix}
      \tilde{\pmb{g}}\\ \tilde{\pmb{h}}
    \end{bmatrix}^{\rm H}H(z_k)\begin{bmatrix}
      \tilde{\pmb{g}}\\ \tilde{\pmb{h}}
    \end{bmatrix}=\sum\limits_{k=1}^{r}\omega_ku(z_k;\widetilde{P},\tilde{q}) \\ \nonumber &=\sum\limits_{k=1}^{r}\omega_k\left(\|F(z_k)\tilde{q}(z_k)-\widetilde{P}(z_k)\|_{\rm F}^2-|\tilde{q}(z_k)|^2\|\hat{e}\|_{\infty,\mathscr{B}}\right) \\\label{eq:quaradtic} &= \sum\limits_{k=1}^{r}\omega_k|\tilde{q}(z_k)|^2\left(\|F(z_k)-\widetilde{R}(z_k)\|_{\rm F}^2-\|\hat{e}\|_{\infty,\mathscr{B}}\right).
  \end{align}
  {However, \eqref{eq:important_equality_rutt_theorem_proof} gives
  \(\|F(z_k)-\widetilde R(z_k)\|_{\rm F}^2-\|\hat e\|_{\infty,\mathscr B}<0\)
  for every \(k\). Since \(\omega_k>0\) and \(\tilde q(z_k)\ne0\), the right-hand side of \eqref{eq:quaradtic} is strictly negative, contradicting \(\mathbf H\succeq0\).}
  Therefore, \(\hat{R} \in \arg\min_{R \in \mathscr{R}_{(s,t)}}\|F-R\|_{\infty,\mathscr{B}}\) is a global solution.

  {It remains to verify that the same points and weights satisfy the dual Kolmogorov condition. Let \(\hat{\pmb a}_{ij}\) and \(\hat{\pmb b}\) be the coefficient vectors of \(\hat p_{ij}\) and \(\hat q\), respectively, and set \(\hat{\pmb c}:=[\hat{\pmb a}^{\rm T},\hat{\pmb b}^{\rm T}]^{\rm T}\), where \(\hat{\pmb a}=[\hat{\pmb a}_{11}^{\rm T},\dots,\hat{\pmb a}_{st}^{\rm T}]^{\rm T}\). Since \(z_k\in{\cal E}(\hat R)\),}
  \[
  {
  \hat{\pmb c}^{\rm H}\mathbf H\hat{\pmb c}
  =\sum_{k=1}^r\omega_k|\hat q(z_k)|^2
  \left(\|F(z_k)-\hat R(z_k)\|_{\rm F}^2-\|\hat e\|_{\infty,\mathscr B}\right)=0.
  }
  \]
  {Because \(\mathbf H\succeq0\), the equality \(\hat{\pmb c}^{\rm H}\mathbf H\hat{\pmb c}=0\) implies \(\mathbf H\hat{\pmb c}=0\). Hence, for any \((G,h)\in\mathscr P_{(s,t)}\times\mathbb P_d\), with coefficient vector \(\pmb c_{G,h}:=[\pmb g^{\rm T},\pmb h^{\rm T}]^{\rm T}\), we have}
  \[
  {
  0=\pmb c_{G,h}^{\rm H}\mathbf H\hat{\pmb c}
  =\sum_{k=1}^{r}\omega_k
  \left(\left\langle F(z_k)h(z_k)-G(z_k),F(z_k)\hat q(z_k)-\hat P(z_k)\right\rangle
  -\overline{h(z_k)}\,\hat q(z_k)\|\hat e\|_{\infty,\mathscr B}\right).
  }
  \]
  {We now identify the \(k\)-th summand in the preceding display. Since \(z_k\in{\cal E}(\hat R)\), \(\|F(z_k)-\hat R(z_k)\|_{\rm F}^2=\|\hat e\|_{\infty,\mathscr B}\). Hence}
  \begin{align*}
  &\left\langle F(z_k)h(z_k)-G(z_k),F(z_k)\hat q(z_k)-\hat P(z_k)\right\rangle
  -\overline{h(z_k)}\,\hat q(z_k)\|\hat e\|_{\infty,\mathscr B} \\
  =& 
  \hat q(z_k)\left(\overline{h(z_k)}\left\langle F(z_k),F(z_k)-\hat R(z_k)\right\rangle
  -\left\langle G(z_k),F(z_k)-\hat R(z_k)\right\rangle\right.\\
  &\qquad\qquad\qquad\left.
  -\overline{h(z_k)}\left\langle F(z_k)-\hat R(z_k),F(z_k)-\hat R(z_k)\right\rangle\right) \\
  =& 
  \hat q(z_k)\left(\overline{h(z_k)}\left\langle \hat R(z_k),F(z_k)-\hat R(z_k)\right\rangle
  -\left\langle G(z_k),F(z_k)-\hat R(z_k)\right\rangle\right).
  \end{align*}
 On the other hand, by the definition \(\hat E(z_k)=(F(z_k)-\hat R(z_k))\hat q(z_k)/\overline{\hat q(z_k)}\) and \(\hat P(z_k)=\hat q(z_k)\hat R(z_k)\), it follows
  \begin{align*}
  &\overline{\left\langle \hat E(z_k),h(z_k)\hat P(z_k)-G(z_k)\hat q(z_k)\right\rangle}  
 =\left\langle h(z_k)\hat P(z_k)-G(z_k)\hat q(z_k),\hat E(z_k)\right\rangle \\
   =&\hat q(z_k)\left(\overline{h(z_k)}\left\langle \hat R(z_k),F(z_k)-\hat R(z_k)\right\rangle
  -\left\langle G(z_k),F(z_k)-\hat R(z_k)\right\rangle\right).
  \end{align*}
 Therefore the \(k\)-th summand is 
  $
  \overline{
  \left\langle \hat E(z_k),h(z_k)\hat P(z_k)-G(z_k)\hat q(z_k)\right\rangle .}
  $
Taking complex conjugates yields \eqref{eq:matrixKolmogorovDual}.
\end{proof}

\subsection{Ruttan's sufficient optimality for  continuum and discrete minimax approximations}\label{subsec:boundcont}
In the previous discussion, we considered  $\mathscr{B}$ as a general compact subset of $\mathbb{C}$. In many practical applications,  one may pursue minimax rational approximants $R(x)$ for matrix-valued functions $F(x)$ defined on a continuum   $\mathscr{B}=\mathscr{D}$ bounded by a simple Jordan curve $\mathscr{L}=\partial\mathscr{D}$. Such approximation problems commonly arise in the study of 
the linear time-invariant system (e.g., \cite{begu:2017,gogu:2021}), frequency-domain multiport rational modeling (see e.g., \cite{demd:2009,gogu:2021,gust:2006,gust:2009,guse:1999}),  computer-aided design of microwave duplexers \cite{trai:2010,trms:2007,zhzy:2025} and nonlinear eigenvalue problems \cite{guti:2017,gupt:2022,limp:2022,saem:2020,zhgz:2026}. The development of efficient approximation methods for such problems remains an active area of research due to their broad relevance in both theoretical and applied contexts.

For the continuum scenarios, we assume $F \in \mathfrak{C}_{A}(\mathscr{D})$, i.e., each entry $f_{ij}(x)$ is analytic in its interior and continuous on the boundary $\mathscr{L}=\partial\mathscr{D}$. As  mentioned in Section \ref{sec_intro} that for any  $R(x)$ with poles outside $\mathscr{D}$, the maximum Frobenius norm principle holds:
$$\max_{x \in \mathscr{D}} \|F(x) - R(x)\|_{\text{F}}^2=\max_{x \in \partial\mathscr{D}} \|F(x) - R(x)\|_{\text{F}}^2  $$
is attained on   $\mathscr{L}=\partial \mathscr{D}$  (see Theorem 1 in \cite{cond:2020}). 
Thus, solving problem \eqref{eq:bestf0} with carefully selected boundary nodes $\mathscr{X} \subset \partial \mathscr{D}$ yields an efficient numerical method for computing minimax rational approximants on $\mathscr{D}$.

To make the following presentation clear, we denote 
\begin{align}\label{eq:Ropts}
R_{\mathscr{D}}&\in \arg\min_{R\in \mathscr{R}_{(s,t)}}\max_{x \in \mathscr{D}} \|F(x) - R(x)\|_{\text{F}}^2,~R_{\mathscr{L}}\in \arg\min_{R\in \mathscr{R}_{(s,t)}}\max_{x \in \mathscr{L}} \|F(x) - R(x)\|_{\text{F}}^2, \\
R_\mathscr{X}&\in \arg\inf_{R\in \mathscr{R}_{(s,t)}}\max_{x \in \mathscr{X}} \|F(x) - R(x)\|_{\text{F}}^2.
\end{align}
Theorem \ref{thm:existence} ensures the existence of $R_{\mathscr{D}}$ and $R_{\mathscr{L}}$ but cannot be applied to the discrete case $R_{\mathscr{X}}$. Moreover, as  pointed out in \cite{this:1993,regu:1983,zhan:2026} that for the scalar-valued case $s=t=1$, $R_{\mathscr{L}}$ may not be the same $R_{\mathscr{D}}$. Indeed, even for the unit disk $\mathscr{D}= \{x \in \mathbb{C} : |x| \leq 1\}$, we can use the example (see \cite{this:1993,regu:1983} and \cite[Section 5]{zhan:2026}) to construct $F(x)=[x^3+x, x^3+x]^{\T}: \mathscr{D}\rightarrow \bbC^2$ with   $n_{11}=n_{21}=0, d=1$ so that 
$$
\inf_{R\in \mathscr{R}_{(2,1)}}\max_{x \in \mathscr{X}} \|F(x) - R(x)\|_{\text{F}}^2\le \|F - R_{\mathscr{L}}\|_{\infty, \mathscr{L}}\approx  6.1966 <\|F - R_{\mathscr{D}}\|_{\infty, \mathscr{D}}\approx 7.4298
$$
where $R_{\mathscr{D}}\approx\left[\frac{0.2993-{\tt i}\cdot 0.0682}{x-1.1194  e^{{\tt i}\cdot 1.1490}},\frac{0.2993-{\tt i}\cdot 0.0682}{x-1.1194  e^{{\tt i}\cdot 1.1490}}\right]^{\T}$ and $R_{\mathscr{L}}=\left[\frac{2/(\sqrt{33}+1)}{x},\frac{2/(\sqrt{33}+1)}{x}\right]^{\T}$.

To determine whether the rational minimax approximation on   the boundary $\mathscr{L}$ or its discrete subset $\mathscr{X} \subseteq \mathscr{L}$ yields an accurate or even exact solution $R_\mathscr{D}$ on $\mathscr{D}$, we derive a sufficient condition based on Ruttan’s global optimality criterion. This is the extension of \cite[Theorem 5.1]{zhan:2026}.

\begin{theorem}\label{thm:sufXtoD}
Let $\mathscr{D} \subset \mathbb{C}$ be a Jordan domain bounded by a   Jordan curve $\mathscr{L} = \partial\mathscr{D}$, and consider $F\in \mathfrak{C}_{A}(\mathscr{D})\setminus \scrR_{(s,t)}$. Assume that the minimax irreducible approximant $R_{\mathscr{D}} \in \scrR_{(s,t)}$  of \eqref{eq:bestf0} satisfies Ruttan's sufficient global optimality \eqref{eq:ruttan_semipositive} with $\mathscr{B}=\mathscr{D}$ {for a certificate \(\{z_k\}_{k=1}^r\subset {\cal E}(R_{\mathscr D})\) and weights \(\{\omega_k\}_{k=1}^r\)}. Then we have 
\begin{itemize}
\item[(i)] $R_{\mathscr{D}}\in \scrR_{(s,t)}$ is also a minimax approximant of $F$ on the boundary $\mathscr{L}$;
\item[(ii)] if $\mathscr{X}=\{x_j\}_{j=1}^m\subset \mathscr{L}$ contains {the certificate points} $\{z_j\}_{j=1}^r\subseteq {\cal E}(R_{\mathscr{D}})\subset \mathscr{L}$ associated with $R_{\mathscr{D}}$ so that Ruttan's sufficient global optimality \eqref{eq:ruttan_semipositive} holds, then $R_{\mathscr{D}}\in \scrR_{(s,t)}$ is also a   minimax approximant of $F$ on $\mathscr{X}$.
\end{itemize}
\end{theorem}
\begin{proof}
Since \(R_{\mathscr D}\) is  pole-free on \(\mathscr D\),  \(F-R_{\mathscr D}\) is analytic in the interior of \(\mathscr D\) and continuous on \(\mathscr D\). By the maximum Frobenius norm principle and $F\in \mathfrak{C}_{A}(\mathscr{D})\setminus \scrR_{(s,t)}$, we know ${\cal E}(R_{\mathscr D})\subseteq \mathscr{L}$ and
$$ 
\|F-R_{\mathscr D}\|_{\infty,\mathscr D}
=\|F-R_{\mathscr D}\|_{\infty,\mathscr L}.
$$
Consequently, every certificate point \(z_k\in{\cal E}(R_{\mathscr D})\) is also an extreme point for the boundary problem. Moreover, the value \(\|\hat e\|_{\infty,\mathscr B}\) entering \(H(x)\) is the same for \(\mathscr B=\mathscr D\) and \(\mathscr B=\mathscr L\). Therefore the same matrix
\(\mathbf H=\sum_{k=1}^r\omega_kH(z_k)\succeq0\) is a Ruttan certificate for \(R_{\mathscr D}\) on \(\mathscr L\). Applying Theorem \ref{thm:ruttan_sufficient} with \(\mathscr B=\mathscr L\) proves (i).

For (ii), since \(\mathscr X\subset\mathscr L\) contains the certificate points \(\{z_k\}_{k=1}^r\), we have
\[
\|F-R_{\mathscr D}\|_{\infty,\mathscr X}
=\|F-R_{\mathscr D}\|_{\infty,\mathscr L}
=\|F-R_{\mathscr D}\|_{\infty,\mathscr D}.
\]
Thus the same points \(z_k\), weights \(\omega_k\), and positive semidefinite matrix \(\mathbf H\) satisfy Ruttan's sufficient condition for the discrete problem with \(\mathscr B=\mathscr X\). A second application of Theorem \ref{thm:ruttan_sufficient} gives that \(R_{\mathscr D}\) is a minimax approximant of \(F\) on \(\mathscr X\).
\end{proof}

\section{Algorithm and foundations for the discrete rational minimax approximations}\label{sec:algorithm}
Based on our discussion in Section \ref{subsec:boundcont}, we now assume that $\mathscr{B}$ is a finite discrete set with \(m \ge \max_{1 \le i \le s,1 \le j \le t}\{n_{ij}+d+2\}\), usually sampled on $\partial \mathscr{D}$. Write \(R:=[r_{ij}]=P/q \in \mathscr{R}_{(s,t)}\) as
\begin{equation}\label{eq:basis_express}
    r_{ij}(x)=\frac{p_{ij}(x)}{q(x)}=\frac{[\psi_0(x),\dots,\psi_{n_{ij}}(x)]\pmb{a}_{ij}}{[\phi_0(x),\dots,\phi_{d}(x)]\pmb{b}},\ \ {\rm for\ some}\ \pmb{a}_{ij} \in \mathbb{C}^{n_{ij}+1},\ \pmb{b} \in \mathbb{C}^{d+1},\ \forall i,j,
\end{equation}
and define the coefficient matrix
\begin{equation}
    \pmb{\Psi}_{ij}=\pmb{\Psi}_{ij}(x_1,\dots,x_m;n_{ij}):=\begin{bmatrix}
        \psi_0(x_1)&\psi_1(x_1)&\dots&\psi_{n_{ij}}(x_1)\\\psi_0(x_2)&\psi_1(x_2)&\dots&\psi_{n_{ij}}(x_2)\\ \vdots & \vdots & \ddots & \vdots\\ \psi_0(x_m)&\psi_1(x_m)&\dots&\psi_{n_{ij}}(x_m)
    \end{bmatrix} \in \mathbb{C}^{m \times \left(n_{ij}+1\right)},\ \forall i,j.
\end{equation}
Analogously, we let \(\pmb{\Phi}=\pmb{\Phi}(x_1,\dots,x_m;d)=[\phi_{j-1}(x_i)]\in \mathbb{C}^{m \times (d+1)}\).
\subsection{A brief review of a dual-based method: {\sf m-d-Lawson}}
As an extension of {\sf d-Lawson} iteration for the scalar-valued case \cite{zhyy:2025,zhha:2025,zhzh:2026}, in \cite{zhzz:2025}, a max-min type dual formulation of the original minimax approximation \eqref{eq:bestf0} is defined  as
\begin{equation}\label{eq:dual_problem}
    \max_{\pmb{w} \in {\cal S}}d(\pmb{w}),
\end{equation}
where \begin{equation}\label{eq:simplex_def}
    {\cal S}=\left\{\pmb{w}=[w_1,\dots,w_m]^{\rm T} \in \mathbb{R}^m : \pmb{w} \ge 0, \sum\limits_{k=1}^mw_{k}=1\right\},
\end{equation}
is the probability simplex and 
\(d(\pmb{w})\)  is the dual function with respect to the dual variable $\bw$ defined by 
\begin{align}\nonumber
    d(\pmb{w})&=\min_{p_{ij} \in \mathbb{P}_{n_{ij}},q \in \mathbb{P}_d, \sum\limits_{k=1}^mw_{k}|q(x_{k})|^2=1}\sum\limits_{k=1}^{m}w_{k}\left(\sum\limits_{i=1}^{s}\sum\limits_{j=1}^{t}|f_{ij}(x_{k})q(x_{k})-p_{ij}(x_{k})|^2\right) \\ \label{eq:lagrange_function}&=\min_{\pmb{a} \in \mathbb{C}^n,\pmb{b} \in \mathbb{C}^{d+1},\|\sqrt{W}\pmb{\Phi}\pmb{b}\|_2=1}\|\sqrt{\pmb{W}_{\otimes}}\left(\pmb{F}\pmb{\Phi}\pmb{b}-\pmb{\Theta}\pmb{a}\right)\|_2^2.
\end{align}
Here \(W={\rm diag}(\pmb{w})\), \(\pmb{W}_{\otimes}=I_{g}\otimes W\), $\pmb{a}=[\pmb{a}_{11}^{\rm T},\dots,\pmb{a}_{s1}^{\rm T},\pmb{a}_{12}^{\rm T},\dots,\pmb{a}_{st}^{\rm T}]^{\rm T},   1\le i\le s, 1\le j\le t,$
\begin{equation}\label{eq:F_Theta_def}
    \pmb{F}=\begin{bmatrix}
        F_{11} \\ \vdots \\ F_{s1} \\ F_{12} \\ \vdots \\ F_{st}
    \end{bmatrix} \in \mathbb{C}^{gm \times m}, \ \ \pmb{\Theta}=\begin{bmatrix}
        \pmb{\Psi}_{11}&&&&&\\&\ddots&&&&\\&&\pmb{\Psi}_{s1}&&&\\&&&\pmb{\Psi}_{12}&&\\&&&&\ddots&\\&&&&&\pmb{\Psi}_{st}
    \end{bmatrix} \in \mathbb{C}^{gm\times n},
\end{equation}
 and  \(F_{ij}={\rm diag}(f_{ij}(x_1),\dots,f_{ij}(x_m)) \in \mathbb{C}^{m \times m}\), \(g=st\), \(n=\sum\limits_{i=1}^s\sum\limits_{j=1}^t\left(n_{ij}+1\right)\).

Note that the maximum of the dual problem in \eqref{eq:dual_problem} is always achievable (see \cite[Proposition 3.1]{zhzz:2025}). Furthermore, a straightforward consequence of the dual problem is the so-called weak duality \cite[Theorem 2.1]{zhzz:2025}:
\begin{equation}\label{eq:weakduality}
\max_{\pmb{w} \in {\cal S}}d(\pmb{w})\le \eta_{\infty,\mathscr{X}}:=\inf_{R \in\scrR_{(s,t)}}\sup_{x\in \mathscr{X}}\|F(x)-R(x)\|_{\F}^2,
\end{equation}
meaning the maximum of the dual problem provides a lower bound of $\eta_{\infty,\mathscr{X}}$. A desired property of the dual problem is the so-called strong duality:
\begin{equation}\label{eq:strongdualitydef}
\max_{\pmb{w} \in {\cal S}}d(\pmb{w})= \eta_{\infty,\mathscr{X}}.
\end{equation}
When strong duality holds and \eqref{eq:bestf0} has a solution, the optimal approximation error $\eta_{\infty,\mathscr{X}}$ is guaranteed to be attained by a rational approximant $\hat R$, which can in principle be obtained by solving the convex dual problem \eqref{eq:dual_problem}.  This assertion is understood in the sense of the scalar Theorem 4.3 of \cite{zhyy:2025}: solvability of the original discrete minimax problem together with   strong duality leads to Ruttan's certificate, and hence to a recoverable global minimax approximant. For a concrete optimizer of the dual problem, the quotient produced from the associated pair \((\hat P,\hat q)\) must still be well defined on \(\mathscr X\), i.e., \(\hat q(x_j)\ne0\) for all \(x_j\in\mathscr X\), in order to evaluate the residuals and recover \(\hat R=\hat P/\hat q\). However, it should be noted that this favorable condition \eqref{eq:strongdualitydef} is not universally satisfied for arbitrary rational approximation problems (see \cite{zhzz:2025} for more discussion on this issue).

To see the solution structure of the dual problem \eqref{eq:dual_problem},  for \(\pmb{w} \ge 0\) and \(W={\rm diag}(\pmb{w})\), we introduce the \(\pmb{w}\)-inner product (positive semidefinite): $$\langle\pmb{y},\pmb{z}\rangle_{W}=\pmb{y}^{\rm H}W\pmb{z}, ~{\rm and}~\|\pmb{y}\|_{W}=\sqrt{\langle\pmb{y},\pmb{y}\rangle_{W}},$$ which, for simplicity, will also be written as \(\langle\pmb{y},\pmb{z}\rangle_{\pmb{w}}\) and \(\|\pmb{y}\|_{\pmb{w}}\), respectively. The following theorem provides the optimality conditions for the pair \((\pmb{a}(\pmb{w}),\pmb{b}(\pmb{w}))\) to be the solution of \eqref{eq:lagrange_function} at given \(\pmb{w} \in {\cal S}\).
\begin{theorem} \label{thm:dualoptm}
For \(\pmb{w} \in {\cal S}\), with the notations in \eqref{eq:dual_problem}, \eqref{eq:simplex_def}, we have 
\begin{itemize}
\item[(i)] {\rm (\cite[Proposition 3.1]{zhzz:2025})}
$\bc(\bw)=\left[\begin{array}{c} \ba(\bw) \\ \bb(\bw)\end{array}\right] \in \bbC^{n+d+1}$ is a solution of \eqref{eq:dual_problem} if and only if it  {is} an eigenvector of the Hermitian positive semidefinite generalized eigenvalue problem $(A_{\bw},B_{\bw})$ and  $d(\bw)$ is the smallest eigenvalue  satisfying
\begin{equation}\label{eq:dual_GEP}
[A_{\bw}-d(\bw) B_{\bw}] \bc(\bw)=\bzs~ \mbox{ and }~ (\bc(\bw))^{\HH}B_{\bw}\bc(\bw) =1,
\end{equation} 
where
\begin{align}\label{eq:dual_GEPA}
A_{\bw}:&=[-{\bf \Theta}, \BF{\bf \Phi}]^{\HH}\BW[-{\bf \Theta}, \BF{\bf \Phi}]=\left[\begin{array}{cc}{\bf \Theta}^{\HH}\BW{\bf \Theta} & -{\bf \Theta}^{\HH}\BW\BF{\bf \Phi} \\-{\bf \Phi}^{\HH} \BF^{\HH}\BW  {\bf \Theta} & {\bf\Phi}^{\HH}\BF^{\HH}\BW\BF{\bf\Phi}\end{array}\right],\\\label{eq:dual_GEPB}
B_{\bw}:&= [0_{m\times n},{\bf \Phi}]^{\HH} W[0_{m\times n},{\bf\Phi}]=\left[\begin{array}{cc}0_{n\times n} & 0_{n\times (d+1)} \\0_{(d+1)\times n} & {\bf \Phi}^{\HH}W{\bf \Phi} \end{array}\right]\in \bbC^{(n+d+1)\times (n+d+1)};
\end{align}
\item[(ii)] {\rm ({\cite[Corollary 3.1]{zhzz:2025}})} for any \(1 \le i \le s\) and \(1 \le j \le t\), it holds that
    \begin{equation}\label{eq:mdlawsonorthorgnalcontion}
        F_{ij}\pmb{q}-\pmb{p}_{ij} \perp_{\pmb{w}} {\rm span}\left(\pmb{\Psi}_{ij}\right),\ \ \pmb{F}^{\rm H}\left(\pmb{F}\pmb{q}-\pmb{p}\right)-d(\pmb{w})\pmb{q} \perp_{\pmb{w}} {\rm span}\left(\pmb{\Phi}\right),
    \end{equation}
    or equivalently, 
    \begin{equation}\label{eq:mdlawsonorthorgnalcontionb}
    \pmb{F}\pmb{q}-\pmb{p} \perp_{\pmb{W}_{\otimes}} {\rm span}\left(\pmb{\Theta}\right),\ \ \pmb{F}^{\rm H}\left(\pmb{F}\pmb{q}-\pmb{p}\right)-d(\pmb{w})\pmb{q} \perp_{W} {\rm span}\left(\pmb{\Phi}\right),
    \end{equation}
    where  \(\pmb{p}=\pmb{\Theta}\pmb{a}(\pmb{w})=[\pmb{p}_{11}^{\rm T},\dots,\pmb{p}_{st}^{\rm T}]^{\rm T} \in \mathbb{C}^{mg}\) and \(\pmb{q}=\pmb{\Phi}\pmb{b}(\pmb{w}) \in \mathbb{C}^m\)  with  \(\pmb{p}_{ij}=\pmb{\Psi}_{ij}\pmb{a}_{ij} \in \mathbb{C}^m\).
\end{itemize}
\end{theorem}

Let \(\hat{\pmb{w}} \in {\cal S}\) be any global maximizer of the dual problem \eqref{eq:dual_problem} with the corresponding pair \((\pmb{a}(\hat{\pmb{w}}),\pmb{b}(\hat{\pmb{w}}))\) from \eqref{eq:lagrange_function} at \(\hat{\bw}\).  {Assume that the denominator \(\hat q(x)=\pmb{b}(\hat{\pmb{w}})^{\rm T}\bm{\phi}(x)\) satisfies \(\hat q(x_j)\ne0\) for all \(x_j\in\mathscr X\).} Let \(\hat{R}=[\hat{r}_{ij}] \in \mathscr{R}_{(s,t)}\) be the representation of the rational function given by \(\hat{r}_{ij}=\hat{p}_{ij}/\hat{q}=\left(\pmb{a}_{ij}^{\rm T}(\hat{\pmb{w}})\bm{\psi}_{ij}\right)/\left(\pmb{b}(\hat{\pmb{w}})^{\rm T}\bm{\phi}\right)\), and \(\|\hat{e}\|_{\infty,\mathscr{X}}=\max_{x \in \mathscr{X}}\|F(x)-\hat{R}(x)\|_{\rm F}^2\). In 
 \cite[Theorem 4.1]{zhzz:2025}, it shows that  strong duality \eqref{eq:strongdualitydef} holds if
\begin{equation}\label{eq:strongduality}
    {d(\hat{\pmb{w}})}=\|\hat{e}\|_{\infty,\mathscr{X}}.
\end{equation}
Notice that the condition \eqref{eq:strongduality} is computationally more convenient  to verify than \eqref{eq:strongdualitydef} and can also be used as rule to check the accuracy of a computed solution (see Step 4 in Algorithm \ref{alg:mdLawson}). 

In the next two subsections \ref{subsec:strongdualRuttan} and \ref{subsection:lawson_kolmogorov_strong_duality}, we shall show respectively that  
\begin{itemize}
\item[1)] Ruttan's sufficient global optimality \eqref{eq:ruttan_semipositive}  is equivalent to  \eqref{eq:strongduality}  {for a dual optimizer whose associated quotient is well defined on \(\mathscr X\)}, and the Hermitian positive semidefinite $H_{\hat\bw}:=A_{\hat\bw}-d(\hat\bw) B_{\hat\bw}\succeq 0$ given in \eqref{eq:dual_GEP} serves as the matrix $\mathbf{H}$ of \eqref{eq:ruttan_semipositive}, and
\item[2)] under  \eqref{eq:strongduality}, the conditions\footnote{The conditions \eqref{eq:mdlawsonorthorgnalcontion} or \eqref{eq:mdlawsonorthorgnalcontionb} are equivalent to $(A_{\bw}-d(\bw) B_{\bw}) \bc(\bw)=\bzs$ in \eqref{eq:dual_GEP}.} \eqref{eq:mdlawsonorthorgnalcontion} is just the Kolmogorov dual criteria \eqref{eq:matrixKolmogorovDual}. 
\end{itemize}
To obtain a global maximizer $\hat \bw$ of  the dual problem \eqref{eq:dual_problem}, \cite{zhzz:2025} proposes the {\sf m-d-Lawson} iteration (Algorithm \ref{alg:mdLawson}) whose convergence is also  developed in \cite[Section 7]{zhzz:2025}.
 {In Step 3 of Algorithm \ref{alg:mdLawson}, the associated rational matrix \(R^{(k)}\) is evaluated directly when the denominator \(\pmb{b}(\pmb{w}^{(k)})^{\rm T}\bm{\phi}(x_\ell)\) is nonzero on the retained nodes.}

\begin{algorithm}[!ht]
\caption{The {\sf m-d-Lawson} iteration \cite{zhzz:2025} for \eqref{eq:bestf0} with \(\mathscr{B}=\mathscr{X}=\{x_{\ell}\}_{\ell=1}^m\)} \label{alg:mdLawson}
\begin{algorithmic}[1]
\renewcommand{\algorithmicrequire}{\textbf{Input:}}
\renewcommand{\algorithmicensure}{\textbf{Output:}}
\REQUIRE Given \(\{(x_{\ell},F(x_{\ell}))\}_{\ell=1}^m\) with \(x_{\ell} \in \mathscr{X}\), a tolerance for strong duality \(\epsilon_r>0\), the maximum number \(k_{\rm maxit}\) of iterations;
\smallskip
\STATE (Initialization) Let \(k=0\); choose \(\bzs<\pmb{w}^{(0)} \in {\cal S}\) and a tolerance \(\epsilon_{\pmb{w}}\) for weights;
\STATE (Filtering) Remove nodes \(x_{\ell}\) with \(w^{(k)}_{\ell}<\epsilon_{\pmb{w}}\);
\STATE compute \(d(\pmb{w}^{(k)})\) and the associated \(R^{(k)}=\frac{\pmb{a}_{ij}(\pmb{w})^{\rm T}\bm{\psi}_{ij}}{\pmb{b}(\pmb{w})^{\rm T}\bm{\phi}}\) from \eqref{eq:lagrange_function};
\STATE (Stop rule) Stop either if \(k \ge k_{maxit}\) or
\begin{equation}\nonumber
    \epsilon(\pmb{w}^{(k)}):=\left|\frac{\zeta(R^{(k)})- {d(\pmb{w}^{(k)})}}{\zeta(R^{(k)})}\right|<\epsilon_r,\ {\rm where}\ \ \zeta(R^{(k)})=\|F(x)-R^{(k)}(x)\|_{\infty,\mathscr{X}};
\end{equation}
\STATE (Updating weights) Update the weight vector \(\pmb{w}^{(k+1)}\) according to 
\begin{equation}\label{eq:Update_rule_lawson}
    w^{(k+1)}_{\ell}=\frac{w^{(k)}_{\ell}\|F(x_{\ell})-R^{(k)}(x_{\ell})\|_{\rm F}^{\beta}}{\sum\limits_{i}w^{(k)}_i\|F(x_i)-R^{(k)}(x_i)\|_{\rm F}^{\beta}},\ \ \ \forall \ell,
\end{equation}
and go to step 2 with \(k=k+1\).
\end{algorithmic}
\end{algorithm}

\subsection{Strong duality and Ruttan's sufficient global optimality}\label{subsec:strongdualRuttan}
Let   \(\hat{\pmb{w}} \in {\cal S}\) be any global maximizer of the dual problem \eqref{eq:dual_problem} with the corresponding pair \((\pmb{a}(\hat{\pmb{w}}),\pmb{b}(\hat{\pmb{w}}))\) from \eqref{eq:lagrange_function}. 
To understand this connection between strong duality \eqref{eq:strongdualitydef} and Ruttan's sufficient global optimality in Theorem \ref{thm:ruttan_sufficient}, the following theorem is crucial.
 
\begin{theorem}\label{thm:strongdualityRuttan}
Let $\hat\bw\in{\cal S}$ be the maximizer of the   dual {problem} \eqref{eq:dual_problem}. 
Ruttan's sufficient condition \eqref{thm:ruttan_sufficient} is satisfied  if and only if 
\begin{equation}\label{eq:strongdualityRuttan}
 {d(\hat{\pmb{w}})}=\|\hat{e}\|_{\infty,\mathscr{X}}:=\max_{x \in \mathscr{X}}\|F(x)-\hat{R}(x)\|_{\rm F}^2,
\end{equation}
{where \(\hat R=\hat P/\hat q\in\scrR_{(s,t)}\) is induced by a pair \((\hat P,\hat q)\) that achieves the minimum \(d(\hat\bw)\) of \eqref{eq:lagrange_function} with \(\bw=\hat\bw\) and satisfies\footnote{The particular dual pair \((\hat P,\hat q)\) need not be irreducible. Since \(\hat q(x_j)\ne0\) on \(\mathscr X\), canceling common factors from \(\hat P\) and \(\hat q\) preserves all values of \(\hat P/\hat q\) on \(\mathscr X\), and hence preserves the discrete objective value and the extreme set. Thus the same discrete approximant admits an equivalent irreducible representation when Ruttan's condition is invoked.}  \(\hat q(x_j)\ne0\) for all \(x_j\in\mathscr X\).}  
\end{theorem}
\begin{proof}
For the sufficiency,  assume \eqref{eq:strongdualityRuttan} holds.  {If the corresponding pair \((\hat P,\hat q)\) is reducible, cancel all common factors of the entries of \(\hat P\) and \(\hat q\), and denote the resulting irreducible representation by \(\hat R=\hat P_{\rm red}/\hat q_{\rm red}\). Since \(\hat q(x_j)\ne0\) for all \(x_j\in\mathscr X\), none of the canceled common factors vanishes on \(\mathscr X\). Hence \(\hat P_{\rm red}(x_j)/\hat q_{\rm red}(x_j)=\hat P(x_j)/\hat q(x_j)\) for all \(x_j\in\mathscr X\), so \(\|\hat e\|_{\infty,\mathscr X}\) and \({\cal E}(\hat R)\) are unchanged. We use this irreducible representation when applying Ruttan's sufficient condition.} By the definition of $H_k(x)$ for $k=1,2,3$ in \eqref{eq:H_x_def},  it holds that 
for the diagonal block ${\bf \Psi}_{ij}$  ($1\le i\le s,~1\le j\le t$) of ${\bf \Theta}$ in \eqref{eq:F_Theta_def},
$$[ {\bf \Psi}_{ij}^{\HH}\be_k\be_k^{\T} {\bf \Psi}_{ij}]_{\ell,u}=\be_{\ell}^{\T} {\bf \Psi}_{ij}^{\HH}\be_k\be_k^{\T}  {\bf \Psi}_{ij}\be_u=\overline{\psi_{\ell-1}(x_k)}\psi_{u-1}(x_k)= [H^{(ij)}_1(x_k)]_{\ell,u},~1 \le \ell,u \le n_{ij}+1,$$ 
and thus 
 \begin{subequations} \label{eq:HwRuttan}
 \begin{align}\label{eq:HwRuttan1}
 &{\bf \Psi}_{ij}^{\HH}\hat W{\bf \Psi}_{ij}=\sum_{k=1}^m \hat w_k   {\bf \Psi}_{ij}^{\HH}\be_k\be_k^{\T} {\bf \Psi}_{ij} =\sum_{k=1}^m \hat w_k  H^{(ij)}_1(x_k),\\
&  -{\bf \Psi}_{ij}^{\HH}F_{ij}\hat W{\bf \Phi}=\sum_{k=1}^m \hat w_k   H^{(ij)}_3(x_k),\\\label{eq:HwRuttan2}
 &{\bf \Phi}^{\HH}F_{ij}^{\HH}\hat WF_{ij}{\bf \Phi}-d(\hat \bw){\bf \Phi}^{\HH} \hat W{\bf \Phi}=\sum_{k=1}^m \hat w_k   H_2(x_k),
 \end{align}
 \end{subequations}
 where we have used ${d(\hat{\pmb{w}})}=\|\hat{e}\|_{\infty,\mathscr{X}}$ in \eqref{eq:strongdualityRuttan} to obtain the last equality. {Since \(\hat q(x_j)\ne0\) on \(\mathscr X\), the extreme set \({\cal E}(\hat R)\) is well defined.} Moreover,  \eqref{eq:strongduality} also ensues 
  \cite[Theorem 5.1]{zhzz:2025}:
 \begin{equation}\label{eq:complement} 
  \hat w_j=0, ~\forall x_j\not \in {\cal E}(\hat R).
  \end{equation}
Consequently,  by the definitions of the positive semidefinite matrix $H_{\hat\bw}$ in \eqref{eq:dual_GEP} and $H_k(x)$ for $k=1,2,3$ in   \eqref{eq:H_x_def},  we conclude  
$$0 \preceq H_{\hat \bw}=\sum_{k=1}^m \hat w_k  H(x_k)=\sum_{k: x_k\in{\cal E}(\hat R)} \hat w_k H(x_k).$$ In other words, we have proved that   \eqref{eq:strongdualityRuttan} (which ensures strong duality \eqref{eq:strongdualitydef}) guarantees Ruttan's sufficient condition \eqref{eq:ruttan_semipositive}; in particular, the  matrix $ H_{\hat \bw}$ serves as the positive semidefinite Hermitian matrix $\mathbf{H}$ in \eqref{eq:ruttan_semipositive}. 

For  the necessity, suppose that Ruttan's sufficient condition \eqref{eq:ruttan_semipositive} holds  for a rational {approximant}  $\hat R =\hat P/\hat q \in  \scrR_{(s,t)}$ with extreme points  $\{z_k\}_{k=1}^r\subseteq {\cal E}(\hat R)$ and $\{\omega_k\}_{k=1}^r$. Theorem {\ref{thm:ruttan_sufficient}} implies that $\hat R$ is the global minimax approximation of \eqref{eq:bestf0} with $\eta_{\infty,\mathscr{X}}=\|\hat e\|_{\infty,\mathscr{X}}$.
Without loss of generality, assume $z_k=x_k$ for $1\le k\le r$; define a vector $\hat \bw \in {\cal S}$ with $ \hat w_k =\omega_k$ if $1\le k\le r$ and $\hat w_k=0$ if $r+1 \le k\le m$. Thus, using the relation \eqref{eq:HwRuttan} and \eqref{eq:ruttan_semipositive}, we know the positive semidefinite matrix $\mathbf{H}$ can be expressed as
$$
\mathbf{H}=A_{\hat \bw } -\eta_{\infty,\mathscr{X}} B_{\hat \bw }\succeq 0.
$$
For \(\hat{R}=\hat P/\hat q=[\hat{r}_{ij}] \in \mathscr{R}_{(s,t)}\), denote by \(\hat{r}_{ij}=\hat{p}_{ij}/\hat{q}=\left(\hat {\ba}_{ij}^{\rm T} \bm{\psi}_{ij}\right)/\left(\hat{\bb}^{\rm T}\bm{\phi}\right)\) and $\hat \ba=[\hat{\ba}_{11}^{\rm T},\dots,\hat {\ba}_{s1}^{\rm T},\hat{\ba}_{12}^{\rm T},\dots,\hat{\ba}_{st}^{\rm T}]^{\rm T}$. Noting $0\ne \hat \bc=\left[\begin{array}{c}\hat \ba \\\hat \bb\end{array}\right] \in \bbC^{n+d+1}$ and by similar calculations as in \eqref{eq:quaradtic}, we have  
\begin{align*}
\hat \bc^{\HH}\left(A_{\hat \bw } -\eta_{\infty,\mathscr{X}} B_{\hat \bw }\right)\hat \bc
&=\sum\limits_{k=1}^{m}\hat w_k \left(\|F(x_k)\hat{q}(x_k)-\hat{P}(x_k)\|_{\rm F}^2-|\hat {q}(x_k)|^2\|\hat{e}\|_{\infty, \mathscr{X}}\right)\\
&=\sum\limits_{k=1}^{r} {\hat w_k}{|\hat{q}(x_k)|^2} \left(\|F(x_k)-\hat{R}(x_k)\|_{\rm F}^2-\|\hat{e}\|_{\infty, \mathscr{X}}\right)=0
\end{align*} 
where the last equality follows due to $\hat w_k=0$ if $r+1 \le k\le m$ and  $\|F(x_k)-\hat{R}(x_k)\|_{\rm F}^2=\|\hat{e}\|_{\infty, \mathscr{X}}$ for $x_1,\dots,x_k\in  {\cal E}(\hat R)$. Since $A_{\hat \bw } -\eta_{\infty,\mathscr{X}} B_{\hat \bw }\succeq 0$, it implies that  $\eta_{\infty,\mathscr{X}} =\|\hat{e}\|_{\infty, \mathscr{X}}$ is the smallest eigenvalue of the pencil $(A_{\hat \bw},  B_{\hat \bw})$ and $\hat \bc$ is the associated eigenvector. Note ${\hat  \bc}^{\HH}B_{\hat \bw}\hat  {\bc}=\sum_{k=1}^r \hat w_k |\hat q(x_k)|^2\ne 0$ {because \(\hat R\) is evaluable at the certificate extreme points and the weights are positive}. Normalize $(\hat P,\hat q)$ using $\hat \bw \in {\cal S}$ so that $\hat q$ is feasible for the minimization \eqref{eq:lagrange_function} (i.e., normalize $\hat \bc$ to have $\hat \bc^{\HH}B_{\hat \bw}\hat \bc=1$); thus $(\hat P,\hat q)$, by  Theorem  \ref{thm:dualoptm}, achieves the minimum $d(\hat \bw)$  of \eqref{eq:lagrange_function} with $\bw=\hat \bw$.  As the minimum $d(\hat \bw)$ of \eqref{eq:lagrange_function} is also the smallest eigenvalue of the pencil $(A_{\hat \bw},  B_{\hat \bw})$, we have $d(\hat \bw)=\eta_{\infty,\mathscr{X}}=\|\hat{e}\|_{\infty, \mathscr{X}}$ and  {weak} duality  \eqref{eq:weakduality} further ensures $\hat \bw$ is a global maximizer of the dual problem \eqref{eq:dual_problem}.  
\end{proof}

Unlike Ruttan's sufficient condition \eqref{eq:ruttan_semipositive}, which is impractical to verify numerically, the condition \eqref{eq:strongduality} offers a computationally tractable alternative  {once the associated rational matrix is well defined on \(\mathscr X\)}. In practice, we can efficiently solve the dual problem \eqref{eq:dual_problem} using established methods such as the {\sf m-d-Lawson} iteration, then numerically verify \eqref{eq:strongduality} at the solution. Moreover, the duality gap naturally provides an effective termination criterion for the {\sf m-d-Lawson} algorithm (see Step 4 in Algorithm \ref{alg:mdLawson}).

\subsection{{\sf m-d-Lawson} meets Kolmogorov’s condition}\label{subsection:lawson_kolmogorov_strong_duality}

In this subsection, we further show that under  \eqref{eq:strongduality}, if \(\hat{\pmb{w}} \in {\cal S}\) is the global maximizer of the dual problem \eqref{eq:dual_problem}, the optimality condition in \eqref{eq:mdlawsonorthorgnalcontion} is indeed the Kolmogorov dual criteria \eqref{eq:matrixKolmogorovDual}. To see this, let $$\hat{\pmb{p}}_{ij}=[\hat{p}_{ij}(x_1),\dots,\hat{p}_{ij}(x_m)]^{\rm T}, ~\hat{\pmb{p}}=[\hat{\pmb{p}}^{\rm T}_{11},\dots,\hat{\pmb{p}}^{\rm T}_{s1},\hat{\pmb{p}}^{\rm T}_{12},\dots,\hat{\pmb{p}}^{\rm T}_{st}]^{\rm T}, ~  \hat{\pmb{q}}=[\hat{q}(x_1),\dots,\hat{q}(x_m)]^{\rm T},$$ where \((\hat{P},\hat{q})\) is associated with \(\hat{\pmb{w}}\).

Consider \(\pmb{F}\hat{\pmb{q}}-\hat{\pmb{p}} \perp_{\widehat{\pmb{W}}_{\otimes}} {\rm span}\left(\pmb{\Theta}\right)\) first where   \(\widehat{W}={\rm diag}\left(\hat{\pmb{w}}\right)\ {\rm and}\ \widehat{\pmb{W}}_{\otimes}=I_{g}\otimes \widehat{W}\). Based again on \eqref{eq:complement},  for any $ 1\le i\le s, 1\le j\le t, 0\le  \ell\le  n_{ij}$, one can write it as
\begin{align}\label{eq:mdl_orthorgnal_item1} 
  \sum\limits_{k=1}^{m}\hat{w}_{k}\overline{\left(f_{ij}(x_{k})\hat{q}(x_{k})-\hat{p}_{ij}(x_{k})\right)}\psi_{\ell}(x_{k})=\sum\limits_{k:x_{k} \in {\cal E}(\hat{R})}\hat{w}_{k}\overline{\left(f_{ij}(x_{k})\hat{q}(x_{k})-\hat{p}_{ij}(x_{k})\right)}\psi_{\ell}(x_{k})=0.
\end{align}
Then for any \(1\le i_0 \le  s,\ 1\le j_0 \le t\) and \(0 \le \ell_0 \le n_{ij}\), the above   \eqref{eq:mdl_orthorgnal_item1} with $i=i_0, j=j_0$ and $\ell=\ell_0$ is equivalent to \eqref{eq:matrixKolmogorovDual} with choice
\begin{equation}\nonumber
    h(x)\equiv 0\ {\rm and}\ G(x)=[g_{ij}(x)]\ {\rm with}\ g_{ij}(x)=\begin{cases}
        -\psi_{\ell_0}(x),\ \ &i=i_0\ {\rm and}\ j=j_0,\\
        0,& {\rm otherwise}.
    \end{cases}
\end{equation}

Now, consider the second condition \(\pmb{F}^{\rm H}\left(\pmb{F}\hat{\pmb{q}}-\hat{\pmb{p}}\right)-d(\pmb{w})\hat{\pmb{q}} \perp_{\hat{\pmb{w}}} {\rm span}\left(\pmb{\Phi}\right)\). Using  \(d(\hat{\pmb{w}})=\|\hat{e}\|_{\infty,\mathscr{X}}=\|\hat{e}(x)\|_{\rm F}^2\ \forall x \in {\cal E}(\hat{R})\), for any $0\le \ell\le d$, one  has
\begin{align}\nonumber
    0 &= \sum\limits_{k=1}^{m}\hat{w}_{k}\left(\sum\limits_{i=1}^{s}\sum\limits_{j=1}^{t}f_{ij}(x_{k})\overline{\left(f_{ij}(x_{k})\hat{q}(x_{k})-\hat{p}_{ij}(x_{k})\right)}-\|\hat{e}\|_{\infty,\mathscr{X}}\cdot\overline{\hat{q}(x_{k})}\right)\phi_{\ell}(x_{k}) \\ \nonumber &=\sum\limits_{k:x_{k} \in {\cal E}(\hat{R})}\hat{w}_{k}\left(\sum\limits_{i=1}^{s}\sum\limits_{j=1}^{t}f_{ij}(x_{k})\overline{\left(f_{ij}(x_{k})\hat{q}(x_{k})-\hat{p}_{ij}(x_{k})\right)}-\|\hat{e}(x_{k})\|^2_{\rm F}\cdot\overline{\hat{q}(x_{k})}\right)\phi_{\ell}(x_{k}) \\ \nonumber &= \sum\limits_{k:x_{k} \in {\cal E}(\hat{R})}\hat{w}_{k}\cdot{\rm tr}\left(\left(F(x_{k})\hat{q}(x_{k})-\hat{P}(x_{k})\right)^{\rm H}\left(\hat{R}(x_{k})\phi_{\ell}(x_{k})\right)\right) \\ \label{eq:mdlorthorgnalitem2} &=\sum\limits_{k:x_{k} \in {\cal E}(\hat{R})}\hat{w}_{k}\left\langle\hat{E}(x_{k}),\phi_{\ell}(x_{k})\hat{P}(x_{k})\right\rangle,
\end{align}
where the second last equality is due to \(\|\hat{e}(x_{k})\|_{\rm F}^2=\sum\limits_{i=1}^{s}\sum\limits_{j=1}^{t}|f_{ij}(x_{k})-\hat{r}_{ij}(x_{k})|^2\) with a straightforward computation. Consequently, \eqref{eq:mdlorthorgnalitem2} corresponds to \eqref{eq:matrixKolmogorovDual} with 
\begin{equation}\nonumber
    h(x)=\phi_{\ell}(x) \in \mathbb{P}_{d}\ {\rm and}\ G(x)=[g_{ij}(x)]\ {\rm with}\ g_{ij}(x)\equiv 0,\ i=1,\dots,s,j=1,\dots,t.
\end{equation}

From the above analysis, we find that, under \eqref{eq:strongduality}, the solution computed by {\sf m-d-Lawson} automatically identifies the associated support extreme points \(x_{k}\) and the corresponding weights \(\omega_{k}\) in the Kolmogorov dual criteria \eqref{eq:matrixKolmogorovDual}.

\section{Concluding remarks}\label{sec:conclude}
In this paper, we have presented a theoretical framework for the rational minimax approximation of matrix-valued functions with a common denominator. Our treatments extend the classical scalar case while addressing new technical challenges arising from the matrix-valued Frobenius norm error and the shared denominator structure. By adapting Walsh's classical scalar existence proof \cite{wals:1931} via a denominator-preserving diagonal procedure, we established the existence of minimax approximants when the underlying domain $\mathscr{B}$ of \eqref{eq:bestf0} is dense in itself. We also developed necessary local optimality conditions analogous to Kolmogorov's criteria (Theorems \ref{thm:kolmogorov_primal} and \ref{thm:matrixKolmogorov_dual}) and sufficient conditions of Ruttan type for global optimality (Theorem \ref{thm:ruttan_sufficient}), tailored for the matrix-valued context to characterize global optimal solutions. 

The developed   optimality conditions, on the one hand, offer a rigorous theoretical foundation for matrix-valued rational approximation, and on the other hand, serve as a practical bridge between theoretical insights and computational applications. In particular, by relying on Ruttan's sufficient condition (Theorem \ref{thm:ruttan_sufficient}) for global optimality, we further clarified a sufficient condition (Theorem \ref{thm:sufXtoD}) to obtain rational minimax approximants for functions analytic in the interior of a   continuum via discrete rational minimax problems with finite nodes sampled on its boundary.  Furthermore, to solve the discrete case, we introduced a recently developed dual-based algorithm {\sf m-d-Lawson} \cite{zhzz:2025} and connected these optimality conditions with the dual framework of the {\sf m-d-Lawson} iteration: we showed in Theorem \ref{thm:strongdualityRuttan} that Ruttan's sufficient global optimality  is equivalent to \eqref{eq:strongdualityRuttan}, which is closely related to strong duality  \eqref{eq:strongdualitydef}, and the orthogonality conditions produced by the dual formulation are precisely Kolmogorov dual criteria under \eqref{eq:strongdualityRuttan}. This gives a practical interpretation of the duality gap used in the {\sf m-d-Lawson} method and explains how the method identifies support extreme points and their weights.

We conclude this paper by highlighting several issues that warrant further investigation. First, as in the scalar case, a full characterization of global optimality for matrix-valued rational minimax approximation over general compact sets remains an open question. Additionally, fundamental issues such as uniqueness, stability under boundary sampling, and convergence from discrete to continuum approximations require deeper theoretical analysis. Finally, developing efficient computational methods for large-scale matrix-valued problems presents significant practical challenges that have yet to be fully addressed.
 
 {\small
  \def\noopsort#1{}\def\l{\char32l}\def\v#1{{\accent20 #1}}
  \let\^^_=\v\def\hbk{hardback}\def\pbk{paperback}
\providecommand{\href}[2]{#2}
\providecommand{\arxiv}[1]{\href{http://arxiv.org/abs/#1}{arXiv:#1}}
\providecommand{\url}[1]{\texttt{#1}}
\providecommand{\urlprefix}{URL }

}

\end{document}